\documentclass[12pt]{article}
\usepackage{a4wide}
\usepackage{epsfig}
\usepackage{amsmath,amssymb,amsthm}
\usepackage{latexsym}
\usepackage{color}
\usepackage{enumitem}
\usepackage{authblk}
\usepackage[sort,nocompress]{cite} %citation numbers increase

\def\RR{\mathbb{R}}

\def\ZZ{\mathbb{Z}}

\def\calP{{\cal P}}

\def\nknots{N}
\def\npatches{M}

\def\hmin{h_\mathrm{min}}

\def\pnew{p}
\def\prec{t}
\def\rbar{r}
\def\Cgen{{\mathfrak{C}}}
\def\Cupper{{\mathfrak{C}}}
\def\Ddom{\Omega}
\def\Pdom{\widetilde{\Omega}}
\def\dpartial{\tilde{\partial}}

\def\knots{\Xi}
\def\dknots{{\bf \Xi}}

\newcommand{\C}[1]{ \frac{h}{2\sqrt{(#1+1)(#1+2)}}}

\newtheorem{theorem}{Theorem}
\newtheorem{lemma}{Lemma}

\newtheorem{corollary}{Corollary}

\theoremstyle{remark}\newtheorem{remark}{Remark}
\theoremstyle{remark}\newtheorem{example}{Example}

\usepackage[para]{footmisc}

\begin{document}

\title{Explicit error estimates for spline approximation of arbitrary smoothness in isogeometric analysis}
%\title{Error estimates in IGA with explicit constants}
%
\author[]{Espen Sande\thanks{sande@mat.uniroma2.it}$~$}
\author[]{Carla Manni\thanks{manni@mat.uniroma2.it}$~$}
\author[]{Hendrik Speleers\thanks{speleers@mat.uniroma2.it}}
\affil[]{\small Department of Mathematics, University of Rome Tor Vergata, Italy}

\maketitle

\begin{abstract}
In this paper we provide a priori error estimates with explicit constants for both the $L^2$-projection and the Ritz projection onto spline spaces of arbitrary smoothness defined on arbitrary grids. This extends the results recently obtained for spline spaces of maximal smoothness.
The presented error estimates are in agreement with the numerical evidence found in the literature that smoother spline spaces exhibit a better approximation behavior per degree of freedom, even for low smoothness of the functions to be approximated.
First we introduce results for univariate spline spaces, and then we address multivariate tensor-product spline spaces and isogeometric spline spaces generated by means of a mapped geometry, both in the single-patch and in the multi-patch case.
\end{abstract}
%%%%%%%%%%%%%

\section{Introduction}
Spline approximation is a classical topic in approximation theory; we refer the reader to the book \cite{Schumaker2007} for an extended bibliography. Moreover, it has recently received a renewed interest within the emerging field of isogeometric analysis (IGA); see the book \cite{Cottrell:09}. In this context, a priori error estimates in Sobolev (semi-)norms and corresponding projectors for suitably chosen spline spaces are important. 

Classical a priori error estimates for spline approximation are explicit in the grid spacing but hide the influence of the smoothness and the degree of the spline space. Such structure, however, is not sufficient for the IGA environment.
In particular, IGA allows for a rich assortment of refinement strategies \cite{Cottrell:09}, combining grid refinement ($h$) and/or degree refinement ($p$) with various interelement smoothness ($k$). To fully exploit the benefits of the so-called $h$-$p$-$k$ refinement, it is necessary to understand how all the parameters involved (i.e., the grid spacing, the degree, and the smoothness) affect the error estimate.
Furthermore, it is important to unravel the influence of the geometry map in isogeometric approximation schemes, not only for its effect on the accuracy but also because it helps in defining good mesh quality metrics \cite{Engvall:preprint}.

Besides their prominent interest for analyzing convergence under different kinds of refinements, error estimates for approximation in suitable reduced spline spaces play a less evident but still pivotal role in other aspects of IGA discretizations, such as the design of fast iterative (multigrid) solvers for the resulting linear systems \cite{Hofreither:2017,Takacs:2018}. The convergence rate of fast iterative solvers should ideally be independent of all the parameters involved, and so their explicit impact on the estimates is important to understand.

In the context of IGA, the role of the smoothness and the degree in spline approximation has been theoretically investigated for the first time in \cite{Buffa:11}, providing explicit error estimates for spline spaces of smoothness $k$ and degree $p\geq 2k+1$. The important case of maximal smoothness ($k=p-1$) has been recently addressed for uniform grid spacing in \cite{Takacs:2016} and for general grid spacing in \cite{Sande:2019}, where improved error estimates have been achieved as well. The above references all deal with both univariate and multivariate spline spaces.

In this paper we provide a priori error estimates with explicit constants for approximation by spline functions of arbitrary smoothness defined on arbitrary knot sequences. Besides filling the gap of the smoothness that is not yet covered in the literature, our results also improve upon the error estimates in \cite{Buffa:11,Takacs:2016,Sande:2019}. 
The key ingredient to get our results is the representation of the considered Sobolev spaces and the approximating spline spaces in terms of integral operators described by suitable kernels \cite{Pinkus:85}.
By using this representation we provide an abstract framework that converts explicit constants in polynomial approximation to explicit constants in spline approximation.
We consider error estimates for both univariate and multivariate spline spaces, and we also allow for a mapped geometry.
After a short description of some preliminary notation, the main theoretical contributions and the structure of the paper are outlined in the next subsections.

\subsection{Preliminary notation}
For $k\geq0$, let $C^k[a,b]$ be the classical space of functions with continuous derivatives of order $0,1,\ldots,k$ on the interval~$[a,b]$.
We further let $C^{-1}[a,b]$ denote the space of bounded, piecewise continuous functions on $[a,b]$ that are discontinuous only at a finite number of points.

Suppose $\knots:= (\xi_0,\ldots,\xi_{\nknots+1})$ is a sequence of (break) points
such that
\begin{equation*}%\label{knots}
a=:\xi_0 < \xi_1 < \cdots < \xi_{\nknots} < \xi_{\nknots+1}:= b,
\end{equation*}
and let
\begin{equation}\label{eq:hmax}
h:=\max_{j=0,\ldots,\nknots} (\xi_{j+1}-\xi_j).
\end{equation}
Moreover, set $I_j := [\xi_j,\xi_{j+1})$,
$j=0,1,\ldots,\nknots-1$, and $I_\nknots := [\xi_\nknots,\xi_{\nknots+1}]$.
For any $p \geq 0$, let $\calP_p$ be the space of polynomials of
degree at most $p$. Then, for $-1\leq k\leq p-1$, we define the space $\mathcal{S}^k_{p,\knots}$ of splines of degree $p$ and smoothness $k$ by
\begin{equation*}
\mathcal{S}^k_{p,\knots } := \{s \in C^{k}[a,b] : s|_{I_j} \in \calP_p,\, j=0,1,\ldots,\nknots \},
\end{equation*}
and we set
\begin{equation*}
\mathcal{S}_{p,\knots} := \mathcal{S}^{p-1}_{p,\knots}.
\end{equation*}
With a slight misuse of terminology, we will refer to $\knots$ as knot sequence and to its elements as knots.

For real-valued functions $f$ and $g$ we denote the norm and inner product on $L^2(a,b)$ by
\begin{equation*}
\| f\|^2 := (f,f), \quad (f,g) := \int_a^b f(x) g(x) dx,
\end{equation*}
and we consider the Sobolev spaces
\begin{equation*}
H^r(a,b):=\{u\in L^2(a,b) : \partial^\alpha u \in L^2(a,b),\, \alpha=1,\ldots,r\}.
\end{equation*}
We use the notation $S_p^k: L^2(a,b)\to \mathcal{S}_{p,\knots}^k$ and  $S_p: L^2(a,b)\to \mathcal{S}_{p,\knots}$ for the $L^2$-projector onto spline spaces, while $P_p: L^2(a,b)\to \calP_p$ stands for the $L^2$-projector onto the polynomial space $\mathcal{P}_{p}$.

\subsection{Main results: univariate case}
In this paper we focus on general spline spaces of degree $p$, smoothness $k$, and arbitrary knot sequence $\knots$. We first derive the following (simplified) error estimate:
\begin{equation}\label{eq:low-order-intro}
\|u-S^k_pu\|\leq\left(\frac{e\, h}{4(p-k)}\right)^r\|\partial^ru\|,
\end{equation}
for any $u\in H^r(a,b)$ and all $p\geq r-1$. Here $e$ is Euler's number.
We refer the reader to Remark~\ref{rmk:low-order-simple}, Theorem~\ref{thm:low-order}, and Corollary~\ref{cor:low-order} for sharper results.
We then show that similar error estimates hold for standard Ritz projections and their derivatives; see Remark~\ref{rmk:Ritz-lower-simple} and Corollary~\ref{cor:Ritz-lower}.

The inequality in \eqref{eq:low-order-intro} does not only cover the univariate result from \cite{Buffa:11}, but also improves upon it by allowing any smoothness; in particular, the most interesting cases of highly smooth spline spaces are embraced.
As already pointed out in \cite{Buffa:11}, a simple error estimate like \eqref{eq:low-order-intro} is not able to give a theoretical explanation for the numerical evidence that smoother spline spaces exhibit a better approximation behavior per degree of freedom. On the other hand, the sharper estimate provided in Theorem~\ref{thm:low-order} improves per degree of freedom as the smoothness of the spline spaces increase; see Remark~\ref{rmk:comparison} (and Figure~\ref{fig:comparison-max-r}). 
Even though this does not prove the superior approximation per degree of freedom of smoother spline spaces, the presented error estimates are a step towards a complete theoretical understanding of the numerical evidence found in the literature.
For uniform knot sequences, it has been formally shown in \cite{Bressan:2019} that $C^{p-1}$ spline spaces perform better than $C^0$ and $C^{-1}$ spline spaces in almost all cases of practical interest. A similar approximation behavior per degree of freedom is observed for the Ritz projections; see Remark~\ref{rmk:comparison-Ritz} (and Figure~\ref{fig:comparison-Ritz-max-r}).

For maximally smooth spline spaces, the best known error estimate for the $L^2$-projection is given by
\begin{equation} \label{eq:Sande-intro}
\|u-S_pu\|\leq \left(\frac{h}{\pi}\right)^{r}\|\partial^r u\|,
\end{equation}
for any $u\in H^r(a,b)$ and all $p\geq r-1$.
This estimate has been recently proved in \cite{Sande:2019}. Note that the same error estimate also holds for periodic functions/splines \cite{Sande:2019,Shadrin:90en}, for which it has been shown to be optimal on uniform knot sequences \cite{Floater:per,Pinkus:85,Sande:2019}.

It is easy to see that \eqref{eq:Sande-intro} is sharper than \eqref{eq:low-order-intro} for $k=p-1$. Nevertheless, for fixed $r$, this estimate only ensures convergence in $h$, and not in $p$.
The role of the grid spacing and the degree is made more clear in the following estimate:
\begin{equation}\label{eq:harmonic-intro}
\|u-S_pu\| \leq \left(\frac{2eh(b-a)}{e\pi(b-a)+4h(p+1)}\right)^r\|\partial^ru\|,
\end{equation}
for any $u\in H^r(a,b)$ and all $p\geq r-1$; see Remark~\ref{rmk:smoothL2-harmonic}. For small $r$ compared to $p$, a better estimate is formulated in Remark~\ref{rmk:smoothL2-harmonic-smallr}.
The general result, covering both \eqref{eq:Sande-intro} and \eqref{eq:harmonic-intro}, can be found in Corollary~\ref{cor:smoothL2}. Similar estimates hold for Ritz projections and their derivatives; see Remark~\ref{rmk:smoothRitz-harmonic} and Corollary~\ref{cor:smoothRitz}.
The $p$-dependence has also been strengthened for the arbitrarily smooth case in Corollary~\ref{cor:low-order}.

Motivated by their use in the analysis of fast iterative solvers for linear systems arising from spline discretization methods~\cite{Hofreither:2017}, we also provide error estimates for approximation in suitable reduced spline spaces; see Theorems~\ref{thm:Spi} and~\ref{thm:Spi-tilde}.

\subsection{Main results: multivariate case}
The univariate results can be extended to obtain error estimates for approximation in multivariate isogeometric spline spaces.
As common in the related literature \cite{Bazilevs:2006,BeiraoDaVeiga:2012,Buffa:14}, we  first address standard tensor-product spline spaces, then investigate the effect of single-patch geometries for isogeometric spline spaces, and finally discuss $C^0$ multi-patch geometries.
In all cases we provide a priori error estimates with explicit constants, %taking into account 
highlighting all the actors that play a role in the construction of the considered spline spaces: the knot sequences, the degrees, the smoothness, and the possible geometry map.

For tensor-product spline spaces we provide error estimates for $L^2$ and Ritz projections in Theorems~\ref{thm:tensorL2} and~\ref{thm:tensorRitz}, respectively. In case of single-patch geometries, we do not confine ourselves to the plain isoparametric context which is typical in IGA \cite{Cottrell:09}, i.e., the
same space that generates the geometry is mapped to the physical domain, but we allow for possibly different spaces for the geometry representation and the function approximation. In the first instance, we assume geometric mappings that are sufficiently globally smooth; see Theorem~\ref{thm:mapL2} and Example~\ref{ex:mapRitz,r=2}. Afterwards, we also provide error estimates for mappings generated by more general geometry function classes that include spline spaces and NURBS spaces of arbitrary smoothness; see Theorem~\ref{thm:mapL2-gen}  and Example~\ref{ex:splinemapRitz,r=2}. In this perspective, following the literature \cite{BeiraoDaVeiga:2012,Buffa:14}, we introduce suitable bent Sobolev spaces, so as to accommodate a less smooth setting for the geometry. 
We explicitize the role of the (derivatives of the) geometry map in the constants of the error estimates, both for $L^2$ and Ritz projections.
Finally, to deal with the $C^0$ multi-patch setting, we consider a projector that is local to each of the patches and is closely related to the standard Ritz projector.  
Indeed, since the global isogeometric space is continuous, we cannot directly use standard $L^2$-projectors as local building blocks on the patches. Instead, we choose each of the projectors to be interpolatory on the patch boundaries \cite{Buffa:14,Takacs:2018} so that they can be easily combined into a continuous global projector. We provide explicit error estimates for the new local projectors, which immediately give rise to the desired estimates for the global one; see Example~\ref{ex:splinemapQ,r=2}.

Even though the multivariate results emanate from the univariate ones by following arguments similar to those already presented in the literature, see \cite{Bazilevs:2006,BeiraoDaVeiga:2012,Buffa:14,Hughes:18,Takacs:2018} and references therein, the novelty of the provided error estimates is twofold:
\begin{itemize}
 \item they are expressed in terms of explicit constants and cover arbitrary smoothness;
 \item they hold for a certain (mapped) Ritz projector which is very natural in the context of Galerkin methods.  
\end{itemize}
It is also worthwhile to note that, although the current investigation has been mainly motivated by IGA applications, standard $C^0$ tensor-product finite elements are included as special cases.

\subsection{Outline of the paper}
The remainder of this paper is organized as follows. In Section~\ref{sec:gen} we introduce a general framework for dealing with a priori error estimates in standard Sobolev (semi-)norms for $L^2$ and Ritz projections onto univariate finite dimensional spaces represented in terms of integral operators described by a suitable kernel. Based on these results, error estimates with explicit constants are provided for spline spaces of arbitrary smoothness in Section~\ref{sec:low}, and further investigated for the salient case of spline spaces of maximal smoothness in Section~\ref{sec:max}. Section~\ref{sec:reduced} addresses certain reduced spline spaces which can be of interest in several contexts.
Then, we extend those univariate results to the multivariate setting. Standard tensor-product spline spaces are considered in Section~\ref{sec:tensor}, while  
isogeometric spline spaces defined on mapped (single-patch) geometries are covered in Section~\ref{sec:map}; we provide explicit expressions for all the involved constants.
In Section~\ref{sec:multipatch} we discuss a particular Ritz-type projector and related error estimates for isogeometric spline spaces on $C^0$ multi-patch geometries. Finally, we conclude the paper in Section~\ref{sec:conclusion} by summarizing the main theoretical results.

%%%%%%%%%%%%%
\section{General error estimates}\label{sec:gen}
%%%%%%%%%%%%%
In this section we describe a general framework to obtain error estimates for the $L^2$-projection and the Ritz projection onto spaces defined in terms of integral operators.

\subsection{General framework}
For $f\in L^2(a,b)$, let $K$ be the integral operator
\begin{equation}\label{eq:K}
 K f(x) := \int_a^b K(x,y) f(y) dy.
\end{equation}
As in \cite{Pinkus:85}, we use the notation $K(x,y)$ for the kernel of~$K$. We will in this paper only consider kernels that are continuous or piecewise continuous.
We denote by $K^*$ the adjoint, or dual, of the operator $K$,
defined by
\begin{equation*}
 (f,K^\ast g) = (Kf, g).
\end{equation*}
The kernel of $K^\ast$ is $K^\ast(x,y) = K(y,x)$.

Given any finite dimensional subspace $\mathcal{Z}_0\supseteq\mathcal{P}_0$ of $L^2(a,b)$ and any integral operator $K$, we let $\mathcal{Z}_\prec$ for $\prec\geq 1$ be defined by $\mathcal{Z}_\prec:=\mathcal{P}_0+K(\mathcal{Z}_{\prec-1})$. We further assume that they satisfy the equality
\begin{equation}\label{eq:Xsimpl}
\mathcal{Z}_\prec:=\mathcal{P}_0+K(\mathcal{Z}_{\prec-1}) =\mathcal{P}_0+K^*(\mathcal{Z}_{\prec-1}),
\end{equation}
where the sums do not need to be orthogonal (or even direct).
Moreover, let $Z_\prec$ be the $L^2$-projector onto $\mathcal{Z}_\prec$, and define $\Cgen_{\prec,\rbar}\in\RR$ for $\prec,\rbar\geq 0$ to be
\begin{equation} \label{eq:Cp}
\Cgen_{\prec,\rbar}:=\|(I-Z_\prec)K^\rbar\|.
\end{equation}
Note that $\Cgen_{\prec,0}=1$.
In the case $\prec=0$ and $\rbar=1$ we further define the constant $\Cupper\in\RR$ to be 
\begin{equation} \label{eq:C}
\Cupper:=\max\{\|(I-Z_0)K\|,\|(I-Z_0)K^*\|\}.
\end{equation}
The following inequality is stated in \cite[Lemma~2.1]{Sande:2019}. For completeness we provide a short proof here as well.
\begin{lemma}\label{lem:1simpl}
The constants in \eqref{eq:Cp} and \eqref{eq:C} satisfy
\begin{equation*}%\label{eq:Cp-C0}
\Cgen_{\prec,1}\leq \Cupper, \quad \prec\geq 0.  
\end{equation*}
\end{lemma}
\begin{proof}
For $\prec=0$, this is true by the definitions of $\Cgen_{0,1}$ and $\Cupper$.
For $\prec\geq 1$, we see from \eqref{eq:Xsimpl} that $KZ_{\prec-1}$ maps into the space $\mathcal{Z}_\prec$. Now, since $Z_\prec$ is the best approximation into $\mathcal{Z}_\prec$ we have
\begin{equation*}
\|(I-Z_\prec)K\|\leq \|K(I-Z_{\prec-1})\|=\|(I-Z_{\prec-1})K^*\|.
\end{equation*}
Continuing this procedure gives
\begin{align*}
\|(I-Z_\prec)K\|\leq\begin{cases}
  \|(I-Z_0)K\|, & \prec \text{ even},\\
   \|(I-Z_0)K^*\|, & \prec \text{ odd},
\end{cases}
\end{align*}
and the result again follows from the definitions of $\Cgen_{\prec,1}$ and $\Cupper$.
\end{proof}

Inspired by the idea of \cite[Lemma~1]{Floater:2018} we have the following more general result.
\begin{lemma}\label{lem:2simpl}
The constants in \eqref{eq:Cp} satisfy
\begin{equation*}
\Cgen_{\prec,\rbar}\leq \Cgen_{\prec,s}\Cgen_{\prec-s,\rbar-s},
\end{equation*}
for all $0\leq s\leq\prec,\rbar$.
\end{lemma}
\begin{proof}
Observe that the operator $(I-Z_\prec)K^sZ_{\prec-s}K^{\rbar-s}=0$ since 
$K^sZ_{\prec-s}K^{\rbar-s}f\in \mathcal{Z}_\prec$ for any $f\in L^2(a,b)$. Thus,
\begin{equation*}
\|(I-Z_\prec)K^{\rbar}\| =\|(I-Z_\prec)K^s(I-Z_{\prec-s})K^{\rbar-s}\| \leq \|(I-Z_\prec)K^s\|\,\|(I-Z_{\prec-s})K^{\rbar-s}\|,
\end{equation*}
and the result follows from the definition of $\Cgen_{\prec,\rbar}$. % in \eqref{eq:Cp}.
\end{proof}

Similar to \cite[Theorem~2.1]{Sande:2019} we obtain the following estimate.
\begin{lemma}\label{lem:simple}
The constants in \eqref{eq:Cp} and \eqref{eq:C} satisfy
\begin{equation*}
\Cgen_{\prec,\rbar}\leq \Cgen_{\prec,1} \Cgen_{\prec-1,1}\cdots \Cgen_{\prec-r+1,1}\leq \Cupper^\rbar,
\end{equation*}
for all $\prec\geq r-1$.
\end{lemma}
\begin{proof}
The case $\rbar=1$ is contained in Lemma~\ref{lem:1simpl}. For the first inequality, the cases $r\geq 2$ follow from Lemma~\ref{lem:2simpl} (with $s=1$) and induction on $\rbar$. 
The second inequality then follows from Lemma~\ref{lem:1simpl}.
\end{proof}

In the next subsection we consider a particularly relevant integral operator: the Volterra operator.

\subsection{Error estimates for the Ritz projection}
Let $K$ be the integral operator defined by integrating from the left,
\begin{equation}\label{eq:Kint}
(Kf)(x):=\int_a^xf(y)dy.
\end{equation}
One can check that $K^*$ is integration from the right,
\begin{equation*}
(K^*f)(x)=\int_x^bf(y)dy;
\end{equation*}
see, e.g., \cite[Section~7]{Floater:2018}. 
Note that in this case we have $\|(I-Z_0)K\|=\|(I-Z_0)K^*\|$, and so $\Cupper=\Cgen_{0,1}$.
Moreover, the space $H^r(a,b)$ can be described as
\begin{equation}\label{eq:Hr}
H^r(a,b)=\mathcal{P}_{0} + K(H^{r-1}(a,b))=\mathcal{P}_{0} + K^*(H^{r-1}(a,b))
=\mathcal{P}_{r-1}+K^r(H^0(a,b)),
\end{equation}
with $H^0(a,b)=L^2(a,b)$ and $\mathcal{P}_{-1}=\{0\}$.
Thus, any $u\in H^r(a,b)$ is of the form $u=g+K^rf$ for $g\in \mathcal{P}_{r-1}$ and $f\in L^2(a,b)$. This leads to the following error estimate for the $L^2$-projection.

\begin{theorem}\label{thm:L2}
Let $Z_\prec$ be the $L^2$-projector onto $\mathcal{Z}_\prec$ and assume $\mathcal{P}_{r-1}\subseteq\mathcal{Z}_\prec$. Then, for any $u\in H^r(a,b)$ we have
\begin{equation}\label{ineq:L2}
\|u-Z_\prec u\|\leq  \Cgen_{\prec,r}\|\partial^ru\|.
\end{equation}
\end{theorem}
\begin{proof}
Since $\mathcal{P}_{r-1}\subseteq\mathcal{Z}_\prec$ and using \eqref{eq:Hr}, we have $u=g+K^rf$ for $g\in \mathcal{P}_{r-1}$ and $f\in L^2(a,b)$. Thus,
\begin{equation}\label{ineq:Hr}
\|u-Z_\prec u\|=\|g+K^rf-Z_\prec(g+K^rf)\|= \|(I-Z_\prec) K^rf\| \leq  \Cgen_{\prec,r}\|f\|,
\end{equation}
and the result follows from the identity $\partial^ru=f$.
\end{proof}
\begin{remark}\label{rem:Andrea}
By definition of the operator norm, the constant $\Cgen_{\prec,r}$ is the smallest possible constant such that the last inequality in \eqref{ineq:Hr} holds for all $f\in L^2(a,b)$. We thus see from the above proof that whenever $\mathcal{P}_{r-1}\subseteq\mathcal{Z}_\prec$, the constant $\Cgen_{\prec,r}$ is the smallest possible constant such that \eqref{ineq:L2} holds for all $u\in  H^r(a,b)$.
\end{remark}

\begin{example}\label{ex:L2simple}
From the definition of $\mathcal{Z}_\prec$ in \eqref{eq:Xsimpl}, with $K$ as in \eqref{eq:Kint}, it follows that $\mathcal{P}_{r-1}$ is a subspace of $\mathcal{Z}_\prec$ for any $\prec\geq r-1$.
Hence, Theorem~\ref{thm:L2} and Lemma~\ref{lem:simple} imply that for any $u\in H^r(a,b)$ we have
\begin{equation*} %\label{ineq:L2}
\|u-Z_\prec u\|\leq  \Cgen_{\prec,r}\|\partial^ru\|\leq\Cupper^r \|\partial^ru\|,
\end{equation*}
for all $\prec\geq r-1$.
However, as we shall see in the next section, there are important cases where  $\mathcal{P}_{r-1}\subseteq\mathcal{Z}_\prec$ for some $\prec<r-1$ (e.g., if $\mathcal{P}_{k}\subseteq\mathcal{Z}_0$ with $k\geq 1$). Such cases will be considered in our proof of the error estimate in \eqref{eq:low-order-intro} and the sharper estimates in Section \ref{sec:low}.
\end{example}

We now focus on a different projector which is very natural in the context of Galerkin methods.
For any $q=0,\ldots,\prec$ we define the projector $R_\prec^q: H^q(a,b)\to \mathcal{Z}_\prec$ by
\begin{equation}\label{def:Ritz}
\begin{aligned}
(\partial^{q}R_\prec^q u,\partial^qv) &= (\partial^q u, \partial^q v), \quad &&\forall v\in \mathcal{Z}_\prec,
\\
(R_\prec^qu,g)&=(u,g), &&\forall g\in \mathcal{P}_{q-1}.
\end{aligned}
\end{equation}
We remark that $R_\prec^q$ is the Ritz projector for the $q$-harmonic problem.
Observe that this projector satisfies $\partial^q R_\prec^q=Z_{\prec-q}\partial^q$, where $Z_{\prec-q}$ denotes the $L^2$-projector onto $\mathcal{Z}_{\prec-q}$. 
With the aid of the Aubin--Nitsche duality argument we arrive at the following estimate.
\begin{lemma}\label{lem:RitzL2}
Let $u\in H^q(a,b)$ be given, and let $R_\prec^q$ be the projector onto $\mathcal{Z}_\prec$ defined in \eqref{def:Ritz}. Then, for any $\ell=0,\ldots,q$ we have
\begin{equation*} %\label{ineq:RitzL2}
\|\partial^{\ell}(u-R^q_\prec u)\| \leq  \Cgen_{\prec-q,q-\ell}\|\partial^qu-Z_{\prec-q}\partial^qu\|,
\end{equation*}
for all $\prec\geq q$ such that $\mathcal{P}_{q-\ell-1}\subseteq\mathcal{Z}_{\prec-q}$.
\end{lemma}
\begin{proof}
Let $u\in H^q(a,b)$ be given and define $w$ as the solution to the Neumann problem
\begin{equation*} %\label{eq:Aubin-Nitsche}
\begin{aligned}
(-1)^{q-\ell}\partial^{2(q-\ell)}w&=u-R_\prec^qu,
\\
w^{(q-\ell)}(a)&=w^{(q-\ell)}(b)=\cdots=w^{(2(q-\ell)-1)}(a)=w^{(2(q-\ell)-1)}(b)=0.
\end{aligned}
\end{equation*}
Using integration by parts, $q-\ell$ times, we have
\begin{align*}
\|\partial^{\ell}(u-R^q_\prec u)\|^2&=(\partial^{\ell}(u-R^q_\prec u),\partial^{\ell}(u-R^q_\prec u))=(\partial^{\ell}(u-R^q_\prec u),(-1)^{q-\ell}\partial^{\ell}\partial^{2(q-\ell)}w)
\\
&=(\partial^q(u-R^q_\prec u),\partial^{q} w)=(\partial^q(u-R^q_\prec u),\partial^{q} (w-v)),
\end{align*}
for any $v\in\mathcal{Z}_\prec$, since $(\partial^q(u-R^q_\prec u),\partial^qv)=0$. Using $\|\partial^{\ell}(u-R^q_\prec u)\|=\|\partial^{2q-\ell}w\|$ and the Cauchy--Schwarz inequality, we obtain
\begin{equation*}
\|\partial^{\ell}(u-R^q_\prec u)\|\,\|\partial^{2q-\ell}w\|\leq \|\partial^q(u-R^q_\prec u)\|\,\|\partial^q (w-v)\|.
\end{equation*}
If we let $v=R_\prec^qw$, then Theorem~\ref{thm:L2} implies that
\begin{equation*}
\|\partial^q(w-R^q_\prec w)\|=\|\partial^qw-Z_{\prec -q}\partial^qw\|
\leq \Cgen_{\prec-q,q-\ell} \|\partial^{2q-\ell}w\|,
\end{equation*}
since $\mathcal{P}_{q-\ell-1}\subseteq\mathcal{Z}_{\prec-q}$. Thus,
\begin{equation*}
\|\partial^{\ell}(u-R^q_\prec u)\|
\leq \Cgen_{\prec-q,q-\ell} \|\partial^q(u-R^q_\prec u)\|
=\Cgen_{\prec-q,q-\ell}\|\partial^qu-Z_{\prec-q}\partial^qu\|,
\end{equation*}
which completes the proof.
\end{proof}

Theorem~\ref{thm:L2} in combination with Lemma~\ref{lem:RitzL2} results in a more classical error estimate for the Ritz projection.
\begin{theorem}\label{thm:Ritz}
Let $u\in H^r(a,b)$ be given.
For any $q=0,\ldots,r$, let $R_\prec^q$ be the projector onto $\mathcal{Z}_\prec$ defined in \eqref{def:Ritz}. Then, for any $\ell=0,\ldots,q$ we have
\begin{equation*}
\|\partial^{\ell}(u-R^q_\prec u)\|\leq 
\Cgen_{\prec-q,q-\ell} \Cgen_{\prec-q,r-q}\|\partial^ru\|,
\end{equation*}
for all $\prec\geq q$ such that $\mathcal{P}_{r-q-1}\subseteq\mathcal{Z}_{\prec-q}$ and $\mathcal{P}_{q-\ell-1}\subseteq\mathcal{Z}_{\prec-q}$.
\end{theorem}

\begin{example}
Similar to Example \ref{ex:L2simple}, we observe from the definition of $\mathcal{Z}_{\prec-q}$ in \eqref{eq:Xsimpl}, with $K$ as in \eqref{eq:Kint}, that $\mathcal{P}_{r-q-1}$ and $\mathcal{P}_{q-\ell-1}$ are subspaces of $\mathcal{Z}_{\prec-q}$ for any $\prec$ satisfying $\prec\geq r-1$ and $\prec\geq 2q-\ell-1$, respectively.
Then, Lemma~\ref{lem:RitzL2} and Theorem~\ref{thm:Ritz}, together with Lemma~\ref{lem:simple}, imply the following results for $u\in H^q(a,b)$. For any $\ell=0,\ldots,q$ we have
\begin{equation} \label{ineq:RitzL2simple}
\|\partial^{\ell}(u-R^q_\prec u)\| \leq  \Cgen_{\prec-q,q-\ell}\|\partial^qu-Z_{\prec-q}\partial^qu\|
\leq \Cupper^{q-\ell}\|\partial^qu-Z_{\prec-q}\partial^qu\|,
\end{equation}
for all $\prec\geq \max\{q,2q-\ell-1\}$, and
\begin{equation}\label{ineq:Ritzsimple}
\|\partial^{\ell}(u-R^q_\prec u)\|\leq 
 \Cgen_{\prec-q,q-\ell} \Cgen_{\prec-q,r-q}\|\partial^ru\|
\leq \Cupper^{r-\ell}\|\partial^ru\|,
\end{equation}
for all $\prec\geq \max\{q,r-1,2q-\ell-1\}$.
\end{example}

\begin{example}\label{ex:stab-lower}
Let $q=1$. Then, for any $u\in H^1(a,b)$ and $\prec \geq 1$ we have the error estimates
\begin{align*}
\|u-R^1_\prec u\|&\leq \Cgen_{\prec-1,1}\|\partial u-Z_{\prec -1}\partial u\|\leq \Cgen_{\prec-1,1}\|\partial u\|, 
\\
\|\partial(u-R^1_\prec u)\|&\leq \|\partial u-Z_{\prec -1}\partial u\|\leq\|\partial u\|,
\end{align*}
and the stability estimates
\begin{align}
\|\partial R^1_\prec u\|&= \|Z_{\prec -1}\partial u\|\leq  \|\partial u\|, \label{ineq:stab:a}
\\
\|R^1_\prec u\|&\leq \|u\|+ \|R^1_\prec u-u\|\leq \|u\| + \Cgen_{\prec -1,1}\|\partial u\|. \label{ineq:stab:b}
\end{align}
\end{example}

We end this section with an observation that will be relevant in the case of a multi-patch geometry; see Section~\ref{sec:multipatch}.
\begin{lemma}\label{lem:Stefan}
If $\mathcal{P}_2\subseteq\mathcal{Z}_\prec$ then $R^1_\prec u(a)=u(a)$ and $R^1_\prec u(b)=u(b)$.
\end{lemma}
\begin{proof}
Let $q=1$ and pick $v(x)=(x-a)^2$ in \eqref{def:Ritz}. Then, using integration by parts, we have
\begin{align*}
(\partial R^1_\prec u,\partial v) &= 2(b-a)R^1_\prec u(b)-(R^1_\prec u,\partial^2 v)=2(b-a)R^1_\prec u(b)-(R^1_\prec u,2),
\\
(\partial u, \partial v) &= 2(b-a)u(b)-(u,\partial^2 v) = 2(b-a)u(b)-(u,2),
\end{align*}
and $R^1_\prec u(b)=u(b)$, since $(R^1_\prec u,2)=(u,2)$. Similarly, by picking $v(x)=(b-x)^2$ we obtain $R^1_\prec u(a)=u(a)$.
\end{proof}

%%%%%%%%%%%%%
\section{Spline spaces of arbitrary smoothness}\label{sec:low}
%%%%%%%%%%%%%
In this section we show error estimates, with explicit constants, for spline spaces of arbitrary smoothness defined on arbitrary knot sequences. To do this we make use of a theorem in \cite{Schwab:99} for polynomial approximation.

\begin{lemma}\label{lem:poly}
Let $u\in H^r(a,b)$ be given. For any $p\geq r-1$, let $P_p$ be the $L^2$-projector onto $\mathcal{P}_{p}$. Then,
\begin{equation}\label{ineq:poly}
\|u-P_pu\|\leq \left(\frac{b-a}{2}\right)^r\sqrt{\frac{(p+1-r)!}{(p+1+r)!}}\|\partial^r u\|.
\end{equation}
\end{lemma}
\begin{proof}
This follows from \cite[Theorem~3.11]{Schwab:99} since the $L^\infty$-norm of the weight-function is bounded by $1$. 
\end{proof}

\begin{lemma}\label{lem:discont}
Let $u\in H^r(a,b)$ be given. For any $p\geq r-1$ and knot sequence $\knots$, let $S_p^{-1}$ be the $L^2$-projector onto $\mathcal{S}^{-1}_{p,\knots}$. Then,
\begin{equation*}
\|u-S^{-1}_pu\|\leq \left(\frac{h}{2}\right)^r\sqrt{\frac{(p+1-r)!}{(p+1+r)!}}\|\partial^r u\|.
\end{equation*}
\end{lemma}
\begin{proof}
This follows from Lemma~\ref{lem:poly} applied to each knot interval.
\end{proof}
\begin{example}\label{ex:discont}
For $r=1$ we have 
\begin{equation*} 
\|u-S^{-1}_pu\|\leq \C{p}\|\partial u\|.
\end{equation*}
\end{example}

We are now ready to derive an error estimate for the $L^2$-projection onto an arbitrarily smooth spline space $\mathcal{S}^{k}_{p,\knots}$. 
We start by observing that if $\mathcal{Z}_0=\mathcal{S}^{-1}_{p-k-1,\knots}$ we have
$\mathcal{Z}_{k+1}=\mathcal{S}^k_{p,\knots}$,
for the sequence of spaces in \eqref{eq:Xsimpl}.
Specifically,
\begin{equation*}
\mathcal{S}^k_{p,\knots} = \mathcal{P}_0+K(\mathcal{S}^{k-1}_{p-1,\knots}) = \mathcal{P}_0+K^*(\mathcal{S}^{k-1}_{p-1,\knots}), \quad k\geq0,
\end{equation*}
and from Lemma~\ref{lem:discont} (and Example~\ref{ex:discont}) we deduce that 
\begin{equation} \label{eq:discont}
\Cgen_{0,r} \leq \left(\frac{h}{2}\right)^{r}\sqrt{\dfrac{(p-k-r)!}{(p-k+r)!}},
\quad \Cupper\leq \frac{h}{2\sqrt{(p-k)(p-k+1)}},
\end{equation}
for any $r$ such that $\mathcal{P}_{r-1}\subseteq\mathcal{Z}_0=\mathcal{S}^{-1}_{p-k-1,\knots}$; see Remark~\ref{rem:Andrea}.
We then define the constant $c_{p,k,r}$ for $p\geq r-1$ as follows. 
If $k\leq p-2$, we~let 
\begin{equation*}
c_{p,k,r}:=\left(\frac{1}{2}\right)^{r}\begin{cases}
\left(\dfrac{1}{\sqrt{(p-k)(p-k+1)}}\right)^r, &k\geq r-2,\\[0.5cm]
\left(\dfrac{1}{\sqrt{(p-k)(p-k+1)}}\right)^{k+1}\sqrt{\dfrac{(p+1-r)!}{(p-1+r-2k)!}}, & k<r-2,
\end{cases}
\end{equation*}
and if $k=p-1$, we let
\begin{equation*}
c_{p,p-1,r}:=\left(\frac{1}{\pi}\right)^r.
\end{equation*} 
By combining \cite[Theorem~1.1]{Sande:2019} with Theorem~\ref{thm:L2} (and Example~\ref{ex:L2simple}) we obtain the following error estimate.
\begin{theorem}\label{thm:low-order}
Let $u\in H^r(a,b)$ be given. For any knot sequence $\knots$, let $S^k_p$ be the $L^2$-projector onto $\mathcal{S}^k_{p,\knots}$ for $-1\leq k\leq p-1$. Then,
\begin{equation*}
\|u-S^k_pu\|\leq c_{p,k,r}h^r\|\partial^ru\|,
\end{equation*}
for all $p\geq r-1$.
\end{theorem}
\begin{proof}
For $k=p-1$, this result has been shown in \cite[Theorem~1.1]{Sande:2019}; see inequality \eqref{eq:Sande-intro}. 
Now, let $k\leq p-2$.
For $r\leq k+2$, the result follows from Example~\ref{ex:L2simple} (with $\prec=k+1$) and the bound for $\Cupper$ in \eqref{eq:discont}.
On the other hand, for $r>k+2$, we use Theorem~\ref{thm:L2} (with $\prec=k+1$), since $\mathcal{P}_{r-1}$ is a subspace of $\mathcal{Z}_{k+1}=\mathcal{S}^k_{p,\knots}$ for all $p\geq r-1$.
Then, applying Lemma~\ref{lem:2simpl} (with $\prec=k+1$) and Lemma~\ref{lem:simple} (with $\prec=k+1$) we get
\begin{equation*}
\Cgen_{k+1,r}\leq \Cgen_{k+1,k+1}\Cgen_{0,r-k-1}\leq \Cupper^{k+1}\Cgen_{0,r-k-1},
\end{equation*}
and the bounds in \eqref{eq:discont} complete the proof.
\end{proof}

\begin{remark}\label{rem:sharpness}
%Observe that in 
In the case $p=0$ and $r=1$, the error estimate in Lemma~\ref{lem:poly} reduces to
\begin{equation*}
\|u-P_0u\|\leq \frac{b-a}{2\sqrt{2}}\|\partial u\|,
\end{equation*}
for any $u\in H^1(a,b)$.
The above constant is very close, but not equal, to the sharp constant given by the Poincar\'e inequality:
\begin{equation*}
\|u-P_0u\|\leq \frac{b-a}{\pi}\|\partial u\|.
\end{equation*}
How close \eqref{ineq:poly} is to being sharp for degrees $p\geq 1$ is an open question. However, we would like to highlight that any improvement upon the error estimate in \eqref{ineq:poly} could be used in \eqref{eq:discont}, and in the proof of Theorem \ref{thm:low-order}, to immediately deduce sharper constants for spline approximation.
\end{remark}

\begin{remark}\label{rmk:low-order-simple}
We can bound $c_{p,k,r}$ for $k\leq p-2$ as follows. For $r\leq k+2$, we have
\begin{equation*}
c_{p,k,r}\leq \left(\frac{1}{2(p-k)}\right)^r,
\end{equation*}
while for $r>k+2$, using the Stirling formula (see, e.g., \cite[proof of Corollary~3.12]{Schwab:99}), we get
\begin{equation*}
c_{p,k,r}\leq \left(\frac{1}{2(p-k)}\right)^r \left(\frac{e}{2}\right)^{\frac{(r-k-1)^2}{p-k}}
\leq \left(\frac{e}{4(p-k)}\right)^r.
\end{equation*}
As a consequence, the estimate in Theorem~\ref{thm:low-order} can be simplified to
\begin{equation}\label{eq:L2basic}
\|u-S^k_pu\|\leq \left(\frac{e\, h}{4(p-k)}\right)^r\|\partial^ru\|,
\end{equation}
for all $p\geq r-1$.
This is in agreement with the estimate in \cite[Theorem~2]{Buffa:11}.
\end{remark}

\begin{remark}\label{rmk:comparison}
Numerical experiments reveal that smoother spline spaces exhibit a better approximation behavior per degree of freedom; see, e.g., \cite{Evans:2009}. It was observed in \cite{Buffa:11}, however, that a simple error estimate like \eqref{eq:L2basic} does not capture this behavior properly.
The sharper estimate in Theorem~\ref{thm:low-order} seems to provide a more accurate description of this behavior. Now, let $(a,b)=(0,1)$. Assuming a uniform knot sequence $\knots$ and $h\ll1$, the spline dimension can be measured by
\begin{equation}\label{eq:comparison-dim}
n:=\dim(\mathcal{S}_{p,\knots}^k)=\frac{p-k}{h}+k+1\simeq\frac{p-k}{h}.
\end{equation}
Hence, %for fixed $n$ and fixed $r$, 
the estimate in Theorem~\ref{thm:low-order} can be rephrased as
\begin{equation} \label{eq:dof}
\|u-S^k_pu\|\lesssim c_{p,k,r}\left(\frac{p-k}{n}\right)^r\|\partial^ru\|,
\end{equation}
for all $p\geq r-1$. As illustrated in Example~\ref{ex:comparison} (Figure~\ref{fig:comparison}) and Example~\ref{ex:comparison-max-r} (Figure~\ref{fig:comparison-max-r}), numerical evaluation indicates that
\begin{equation*}
c_{p,k_1,r}(p-k_1)^r < c_{p,k_2,r}(p-k_2)^r, \quad k_1>k_2. 
\end{equation*}
This is in agreement with the numerical evidence found in the literature that, for fixed spline degree, smoother spline spaces have better approximation properties per degree of freedom, even for low smoothness of the functions to be approximated. 
We refer the reader to \cite{Bressan:2019} for a more exhaustive theoretical comparison of the approximation power of spline spaces per degree of freedom in the extreme cases $k=-1,0,p-1$.
\end{remark}

\begin{figure}[t!]
\centering
\includegraphics[scale=0.65]{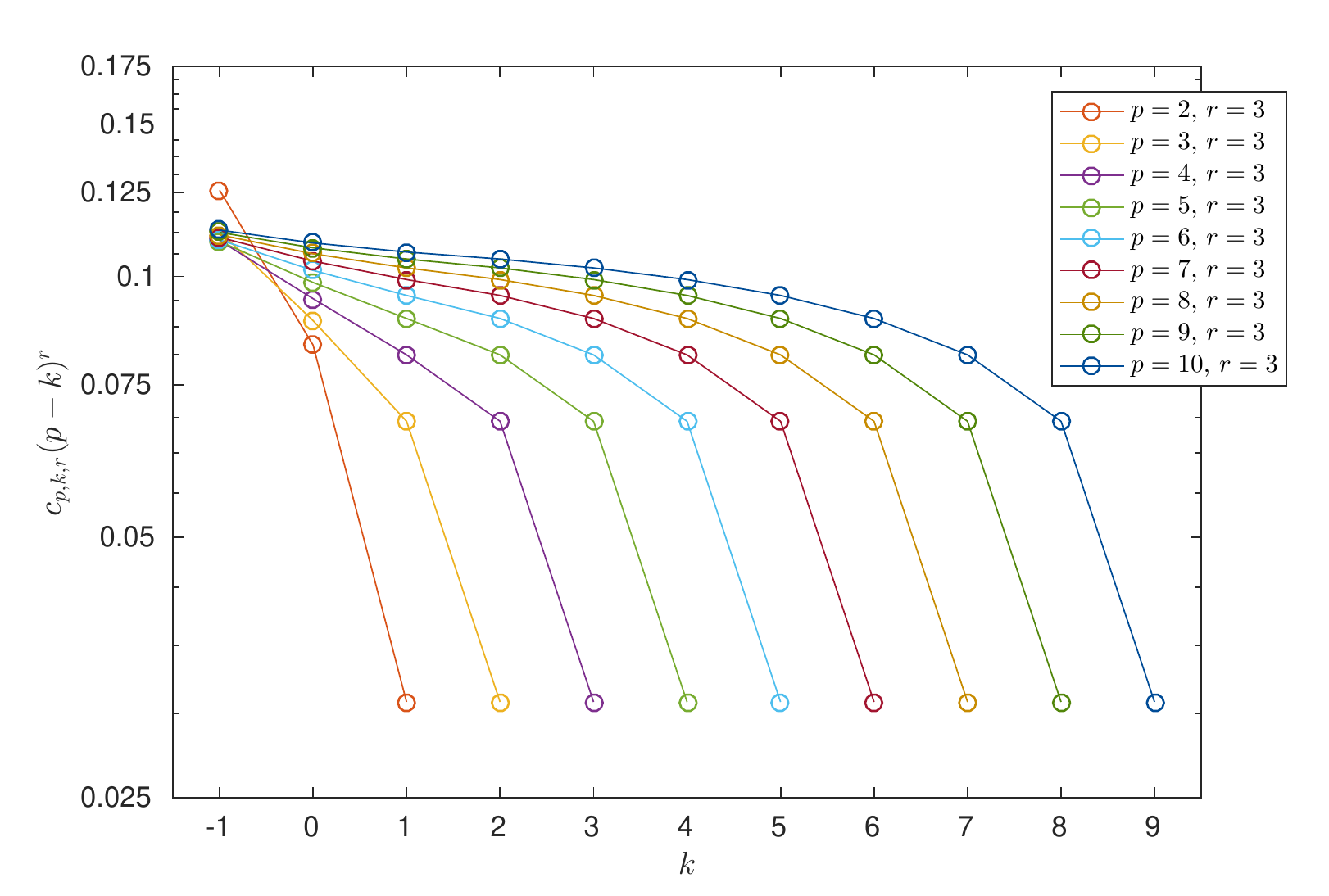}
\caption{Numerical values of $c_{p,k,r}(p-k)^r$ for $r=3$ and different choices of $p\geq r-1$ and $-1\leq k\leq p-1$.
For any fixed $p$, one notices that the values are decreasing for increasing $k$. This means that the smoother spline spaces perform better in the error estimate \eqref{eq:dof} for fixed spline dimension.}\label{fig:comparison}
%\end{figure}
\medskip
%\begin{figure}[t!]
\centering
\includegraphics[scale=0.65]{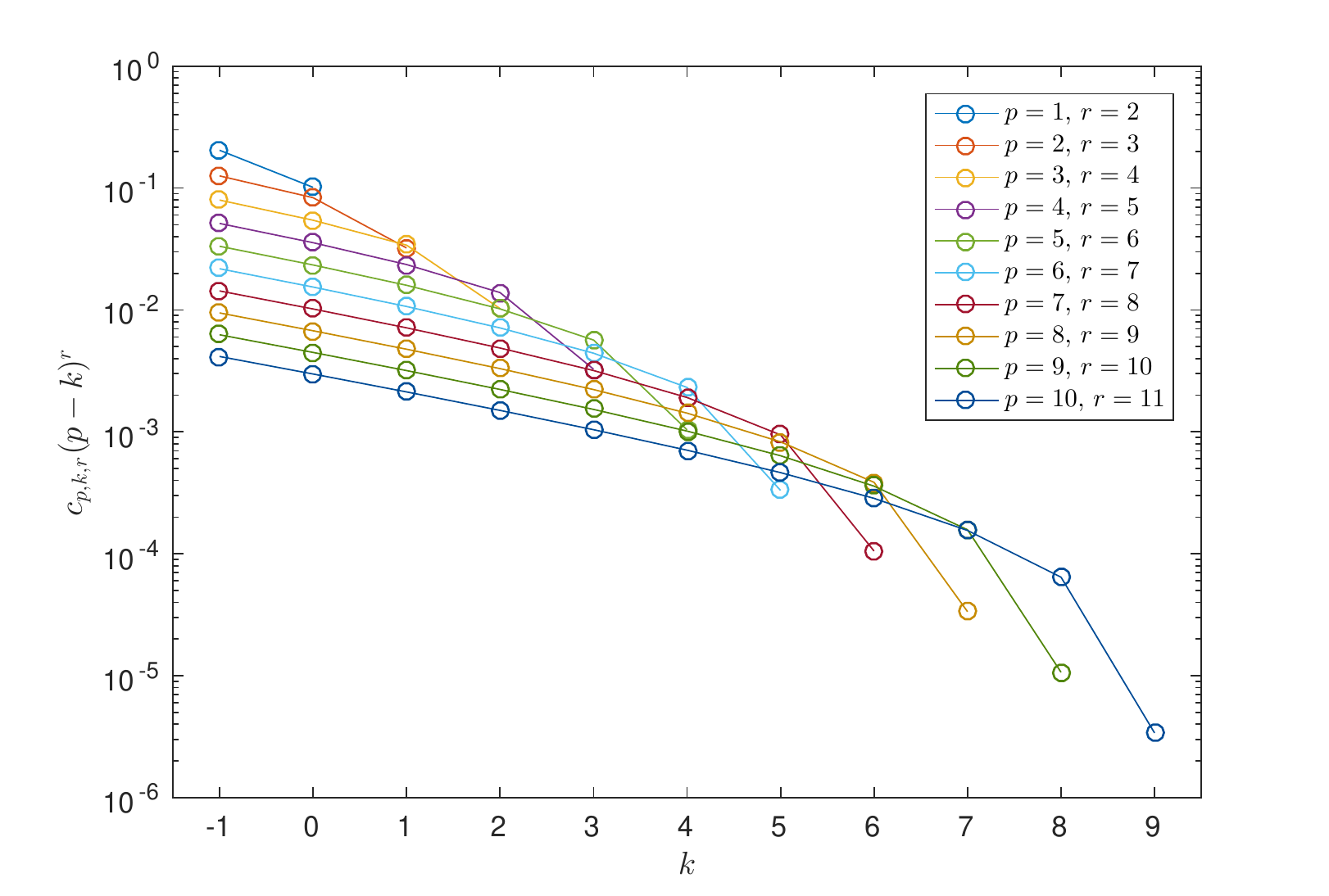}
\caption{Numerical values of $c_{p,k,r}(p-k)^r$ for $r=p+1$ and different choices of $p\geq1$ and $-1\leq k\leq p-1$.
For any fixed $p$, one notices that the values are decreasing for increasing $k$. This means that the smoother spline spaces perform better in the error estimate \eqref{eq:dof} for fixed spline dimension.}\label{fig:comparison-max-r}
\end{figure}

\begin{example}\label{ex:comparison}
Let $r=3$. Figure~\ref{fig:comparison} depicts the numerical values of $c_{p,k,3}(p-k)^3$ for different choices of $p$ and $k$. 
We clearly see that the smallest values are attained for maximal spline smoothness $k=p-1$, namely $c_{p,p-1,3}=(1/\pi)^3\approx 0.0323$.
\end{example}

\begin{example}\label{ex:comparison-max-r}
Consider now the maximal Sobolev smoothness $r=p+1$. Figure~\ref{fig:comparison-max-r} depicts the numerical values of $c_{p,k,p+1}(p-k)^{p+1}$ for different choices of $p$ and $k$. 
For any fixed $p$, one notices that the values are decreasing for increasing $k$, and hence the smallest values are attained for maximal spline smoothness $k=p-1$.
\end{example}

By utilizing Lemma~\ref{lem:poly} once again we can further sharpen the error estimate in Theorem~\ref{thm:low-order}.
Let us now define $C_{h,p,k,r}$ by
\begin{equation}\label{eq:Chp}
C_{h,p,k,r}:=\min\left\{c_{p,k,r}h^r,\left(\frac{b-a}{2}\right)^r\sqrt{\frac{(p+1-r)!}{(p+1+r)!}}\right\},
\end{equation} 
for $p\geq \max\{r-1,k+1\}$. Since $\mathcal{P}_p\subseteq\mathcal{S}^{k}_{p,\knots}$, the following result immediately follows from Lemma~\ref{lem:poly} and Theorem~\ref{thm:low-order}.
\begin{corollary}\label{cor:low-order}
Let $u\in H^r(a,b)$ be given. For any knot sequence $\knots$, let $S^k_p$ be the $L^2$-projector onto $\mathcal{S}^{k}_{p,\knots}$ for $-1\leq k\leq p-1$. Then,
\begin{equation*}%\label{ineq:low-order-2}
\|u-S^k_pu\|\leq C_{h,p,k,r}\|\partial^ru\|,
\end{equation*}
for all $p\geq r-1$.
\end{corollary}
The above corollary shows that $\Cgen_{\prec,r}\leq C_{h,p,k,r}$ for $\mathcal{Z}_\prec=\mathcal{S}^{k}_{p,\knots}$; see Remark~\ref{rem:Andrea}.
Note that for $k=-1$, the constant $C_{h,p,k,r}$ equals $c_{p,k,r}h^r$ for any $p$, $h$ and $r$. However, for large $k$ and $p$ (compared to $1/h$) the second argument in \eqref{eq:Chp} can become smaller than the first. The error estimate in Corollary~\ref{cor:low-order} will in this case then coincide with the error estimate for global polynomial approximation. We will look closer at this case in the next section; 
see in particular Figure~\ref{fig:break-arguments-C}.

In many applications one would be interested in finding a single spline function that can provide a good approximation of all derivatives of $u$ up to a given number $q$. 
Derivative estimates for the $L^2$-projection could be obtained under some quasi-uniformity assumptions on the knot sequence which ensure stability of the $L^2$-projection in the $H^1$ semi-norm; see, e.g., \cite[Theorem~2]{Crouzeix:1987} for such conditions in the case $k=0$. However, these assumptions can be avoided by using a Ritz projection.
As a special case of \eqref{def:Ritz} we define, for any $q=0,\ldots,k+1$, the Ritz projector $R_p^{q,k}: H^q(a,b)\to \mathcal{S}^k_{p,\knots}$ by
\begin{equation}\label{def:Ritz-lower}
\begin{aligned}
(\partial^{q}R_p^{q,k} u,\partial^qv) &= (\partial^q u, \partial^q v), \quad &&\forall v\in \mathcal{S}^k_{p,\knots},
\\
(R_p^{q,k}u,g)&=(u,g), &&\forall g\in \mathcal{P}_{q-1}.
\end{aligned}
\end{equation}
As a consequence of Theorem~\ref{thm:Ritz} we have the following error estimate.
\begin{corollary}\label{cor:Ritz-lower}
Let $u\in H^r(a,b)$  be given.
For any degree $p$, knot sequence $\knots$ and smoothness $-1\leq k\leq p-1$, let $R_p^{q,k}$ be the projector onto $\mathcal{S}^{k}_{p,\knots}$ defined in \eqref{def:Ritz-lower} for $q=0,\ldots,\min\{k+1,r\}$. Then, for any $\ell=0,\ldots,q$, we have
\begin{equation*}
\|\partial^{\ell}(u-R^{q,k}_pu)\| \leq C_{h,p-q,k-q,q-\ell} C_{h,p-q,k-q,r-q}\|\partial^ru\|,
\end{equation*}
for all $p\geq \max\{q,r-1,2q-\ell-1\}$.
\end{corollary}

\begin{remark}\label{rmk:Ritz-lower-simple}
Using the definition of $C_{h,p,k,r}$ together with the argument in Remark~\ref{rmk:low-order-simple}, we can simplify the result in Corollary~\ref{cor:Ritz-lower} as follows. For any $q=0,\ldots,\min\{k+1,r\}$ and $\ell=0,\ldots,q$, we have
\begin{align*}
\|\partial^{\ell}(u-R^{q,k}_pu)\| &\leq c_{p-q,k-q,q-\ell} c_{p-q,k-q,r-q} h^{r-\ell}\|\partial^ru\|
\leq \left(\frac{e\, h}{4(p-k)}\right)^{r-\ell}\|\partial^ru\|,
\end{align*}
for all $p\geq \max\{q,r-1,2q-\ell-1\}$.
Since this estimate is explicit in $h$ and $p$, it is very useful for $h$-$p$ refinement.
\end{remark}

\begin{example}
Similar to Example~\ref{ex:stab-lower} we let $q=1$. Then, for any $u\in H^r(a,b)$ and $k\geq 0$ we have the stability estimates
\begin{align*}
\|\partial R^{1,k}_pu\|&= \|S^{k-1}_{p-1}\partial u\|\leq  \|\partial u\|,
\\
\|R^{1,k}_pu\|&\leq \|u\| +C_{h,p-1,k-1,1}\|\partial u\|
\leq \|u\| +\frac{e\, h}{4(p-k)}\|\partial u\|,
\end{align*}
and, as in Remark~\ref{rmk:Ritz-lower-simple}, the error estimates
 \begin{equation*}
\begin{aligned}
\|u-R^{1,k}_pu\|&\leq C_{h,p-1,k-1,1} C_{h,p-1,k-1,r-1}\|\partial^ru\| 
\leq \left(\frac{e\, h}{4(p-k)}\right)^{r}\|\partial^ru\|,
\\
\|\partial(u-R^{1,k}_pu)\|&\leq C_{h,p-1,k-1,r-1}\|\partial^ru\|
\leq \left(\frac{e\, h}{4(p-k)}\right)^{r-1}\|\partial^ru\|,
\end{aligned}
\end{equation*}
for all $p\geq r-1$.
Thus, $R^{1,k}_pu$ provides a good approximation of both the function $u$ itself, and its first derivative.
\end{example}

\begin{example}
Let $q=2$ and $r=3$. For $R^{2,k}_pu$ to approximate $u\in H^3(a,b)$ in the $L^2$-norm, Corollary~\ref{cor:Ritz-lower} requires the degree to be at least $2q-1=3$, and not $r-1=2$ as one might expect. In view of \eqref{def:Ritz-lower}, this is consistent with the common assumption to solve the biharmonic equation with piecewise polynomials of at least cubic degree for obtaining an optimal rate of convergence in $L^2$; see, e.g., \cite[p.~118]{Strang:73}.
\end{example}

\begin{remark}\label{rmk:comparison-Ritz}
In the spirit of Remark~\ref{rmk:comparison}, the above error estimates for the Ritz projection can also be used to investigate the approximation behavior per degree of freedom. Let $(a,b)=(0,1)$, and assume a uniform knot sequence $\knots$ and $h\ll1$.
Then, keeping the dimension formula \eqref{eq:comparison-dim} in mind, 
%for fixed $n$ and fixed $r,\ell$, 
the first inequality in Remark~\ref{rmk:Ritz-lower-simple} can be rephrased as: for any $q=\ell,\ldots,\min\{k+1,r\}$, we have
\begin{equation} \label{eq:dof-Ritz}
\|\partial^{\ell}(u-R^{q,k}_pu)\|\lesssim c_{p-q,k-q,q-\ell} c_{p-q,k-q,r-q}\left(\frac{p-k}{n}\right)^{r-\ell}\|\partial^ru\|,
\end{equation}
for all $p\geq \max\{q,r-1,2q-\ell-1\}$. 
As illustrated in Example~\ref{ex:comparison-Ritz-max-r} (Figure~\ref{fig:comparison-Ritz-max-r}), numerical evaluation of the constant in \eqref{eq:dof-Ritz} indicates that our error estimate performs better per degree of freedom for smoother spline spaces, not only in the $L^2$ norm but also in more general $H^\ell$ (semi-)norms. 
\end{remark}

\begin{figure}[t!]
\centering
\includegraphics[scale=0.65]{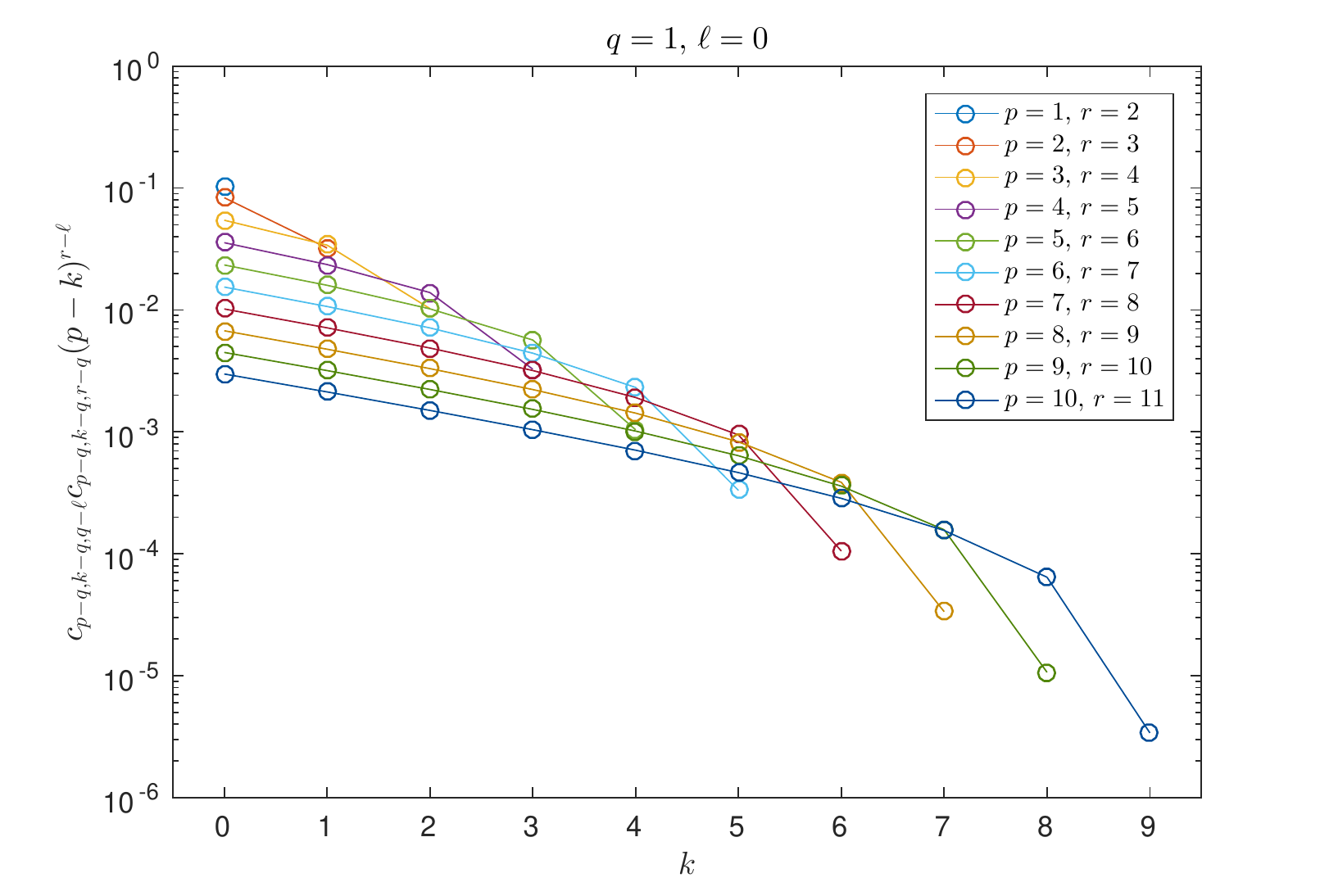} \\
\medskip
\centering
\includegraphics[scale=0.65]{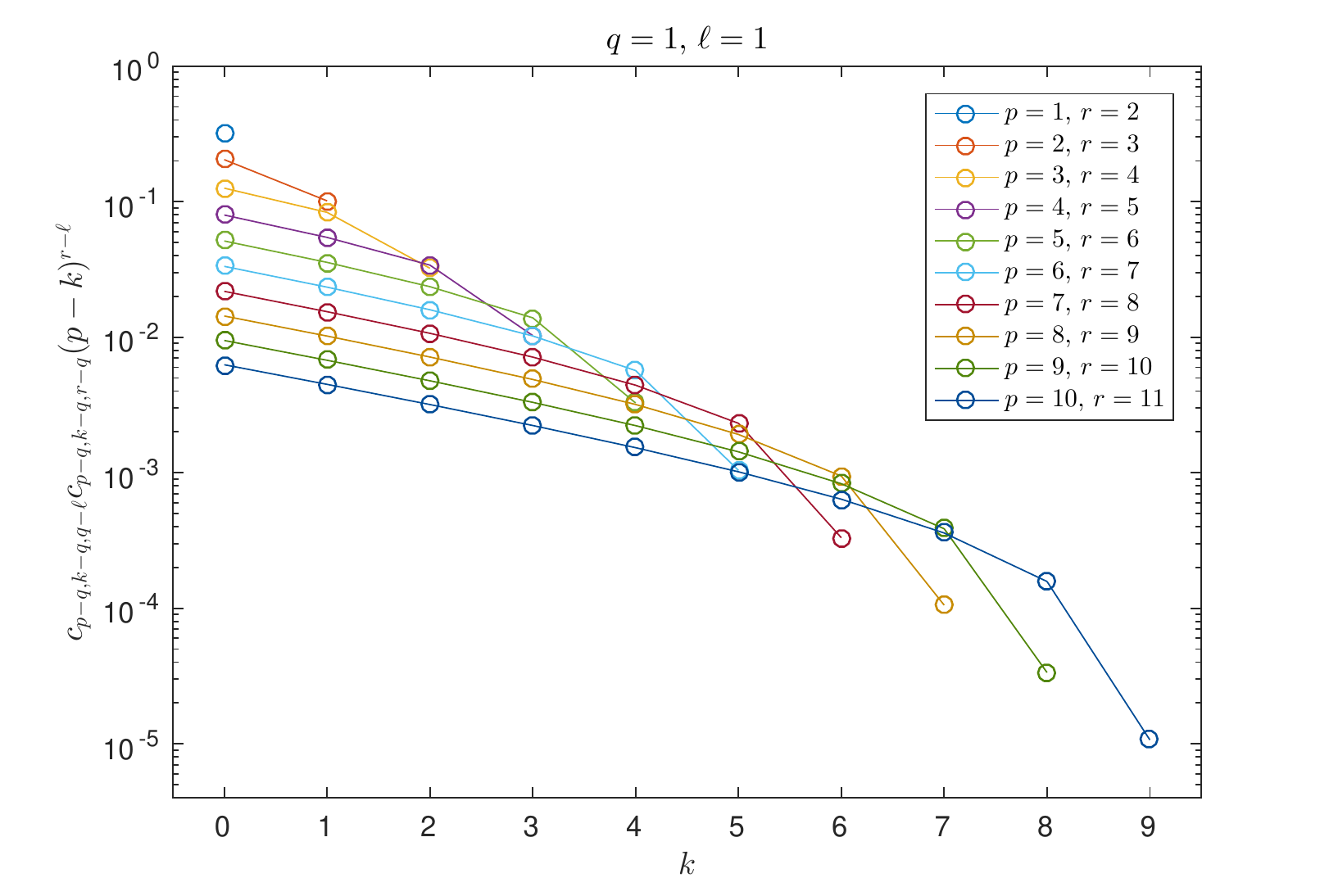}
\caption{Numerical values of $c_{p-q,k-q,q-\ell} c_{p-q,k-q,r-q}(p-k)^{r-\ell}$ for $r=p+1$, $q=1$, $\ell=0,1$, and different choices of $p\geq1$ and $q-1\leq k\leq p-1$.
For any fixed $p$, one notices that the values are decreasing for increasing $k$. This means that the smoother spline spaces perform better in the error estimate \eqref{eq:dof-Ritz} for fixed spline dimension.}\label{fig:comparison-Ritz-max-r}
\end{figure}

\begin{example}\label{ex:comparison-Ritz-max-r}
Let $q=1$ and consider the maximal Sobolev smoothness $r=p+1$. 
Figure~\ref{fig:comparison-Ritz-max-r} depicts the numerical values of $c_{p-q,k-q,q-\ell} c_{p-q,k-q,r-q}(p-k)^{r-\ell}$ for $\ell=0,1$ and different choices of $p$ and $k$. 
For any fixed $p$ and $\ell$, one notices that the values are decreasing for increasing $k$, and hence the smallest values are attained for maximal spline smoothness $k=p-1$.
Since this trend is happening for both $\ell=0,1$, it means that our error estimate performs better per degree of freedom for higher smoothness, in both the $L^2$ and $H^1$ norms, for any fixed $p$. 
Note that the graphs look like the ones in Figure~\ref{fig:comparison-max-r} using the $L^2$-projection. This is not a coincidence because one can check that 
\begin{equation*}
c_{p-1,k-1,1-\ell} c_{p-1,k-1,p} = c_{p-\ell,k-\ell,p+1-\ell},
\end{equation*}
for $0\leq k\leq p-1$ and $\ell=0,1$.
\end{example}

\begin{remark}
The last observation in Example~\ref{ex:comparison-Ritz-max-r} can be generalized as follows. In case of maximal Sobolev regularity $r=p+1$ and $\max\{q-1,2q-\ell-2\}\leq k\leq p-1$, we have
\begin{equation*}
c_{p-q,k-q,q-\ell} c_{p-q,k-q,p+1-q} = c_{p-\ell,k-\ell,p+1-\ell}.
\end{equation*}
\end{remark}

%%%%%%%%%%%%%
\section{Spline spaces of maximal smoothness}\label{sec:max}
%%%%%%%%%%%%%
As mentioned in the introduction, the error estimate in \eqref{eq:Sande-intro} is perfectly suited to study the case of grid refinement with maximally smooth splines. However, it provides almost no information in the case of degree elevation. The best it can tell us is that the error will not get worse as $p$ increases. This is in contrast to the standard error estimates for $C^0$ finite element methods (or the case $k=0$ of \eqref{eq:low-order-intro}), which show clear convergence as $p\to\infty$.
In this section we therefore study error estimates for the space of maximally smooth splines in more detail, and in particular, we investigate the $p$-dependence. The main goal is to obtain various estimates for the full $h$-$p$-$k$ refinement scheme, i.e., as $p\to\infty$ and/or as $h\to 0$ under the constraint $k=p-1$.

Let us define the constant $C_{h,p,r}$ by $C_{h,p,r}:=C_{h,p,p-1,r}$ with $C_{h,p,k,r}$ in \eqref{eq:Chp}, or more explicitly by
\begin{equation}\label{eq:Chp-smooth}
C_{h,p,r}:=\min\left\{\left(\frac{h}{\pi}\right)^r,\left(\frac{b-a}{2}\right)^r\sqrt{\frac{(p+1-r)!}{(p+1+r)!}}\right\},
\end{equation}
for $p\geq r-1$. 
As a generalization of \cite[Corollary~6.3]{Takacs:2016} we obtain the following result.
\begin{corollary}\label{cor:smoothL2}
Let $u\in H^r(a,b)$ be given. 
For any knot sequence $\knots$, let $S_p$ be the $L^2$-projector onto $\mathcal{S}_{p,\knots}$. Then, 
\begin{align}
\|u-S_pu\| &\leq C_{h,p,r}\|\partial^ru\|,\label{ineq:smoothL2-1} 
\\
\|u-S_pu\| &\leq C_{h,p,1}C_{h,p-1,1}\cdots C_{h,p-r+1,1}\|\partial^ru\|,\label{ineq:smoothL2-2}
%\leq (C_{h,p-r+1})^r \|\partial^ru\|,
\end{align}
for all $p\geq r-1$.
\end{corollary}
\begin{proof}
The estimate \eqref{ineq:smoothL2-1} is the case $k=p-1$ of Corollary \ref{cor:low-order}. For \eqref{ineq:smoothL2-2}, we first observe that if $\mathcal{Z}_0=\mathcal{S}_{0,\knots}$ we have
$\mathcal{Z}_{p}=\mathcal{S}_{p,\knots}$
for the sequence of spaces in \eqref{eq:Xsimpl}. From Lemma~\ref{lem:simple} (with $\prec=p$) we then obtain \eqref{ineq:smoothL2-2} for $u\in H^r(a,b)$.
\end{proof}

The first argument in the definition of $C_{h,p,r}$ only depends on $h$ and $r$, while the second argument only depends on $p$ and $r$. Hence, it is clear that the second argument is smaller than the first for large enough $p$ with respect to $h$. This is illustrated in the next examples.
\begin{example} \label{ex:comparison-arguments-C}
Let $r=2$ and $b-a=2$. Then, assuming
\begin{equation*}
p>\frac{\pi}{h},
\end{equation*}
we have
\begin{equation*}
C_{h,p,2} = \frac{1}{\sqrt{p(p+1)(p+2)(p+3)}} < \frac{1}{p^2} < \left(\frac{h}{\pi}\right)^2.
\end{equation*}
In this case the error estimate \eqref{ineq:smoothL2-1} is better than \eqref{eq:Sande-intro}.
\end{example}

\begin{figure}[t!]
\centering
\includegraphics[scale=0.7]{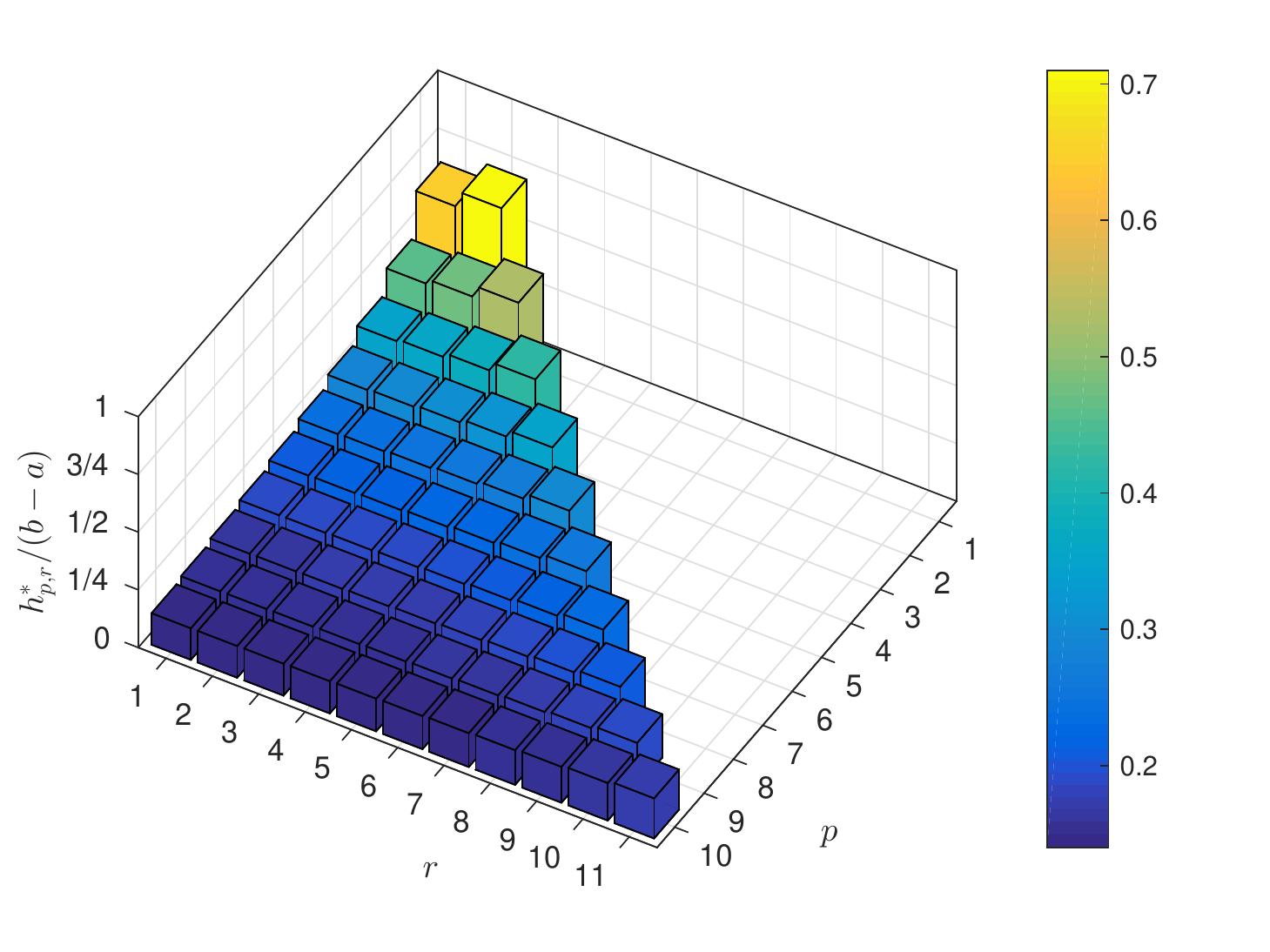}
\caption{The two arguments of $C_{h,p,r}$ in \eqref{eq:Chp-smooth} are equal for $h=h_{p,r}^*$, depicted in normalized form (divided by the interval length $b-a$) for different choices of $r\geq1$ and $p\geq r-1$. Lower values of $h$ will activate the first argument of $C_{h,p,r}$, while higher values of $h$ the second argument of $C_{h,p,r}$.}\label{fig:break-arguments-C}
\end{figure}

\begin{example}\label{ex:break-arguments-C}
Figure~\ref{fig:break-arguments-C} depicts the values $h_{p,r}^*/(b-a)\in[0,1]$ satisfying
\begin{equation*}
\left(\frac{h_{p,r}^*}{\pi}\right)^r=\left(\frac{b-a}{2}\right)^r\sqrt{\frac{(p+1-r)!}{(p+1+r)!}},
\end{equation*}
for different choices of $r$ and $p$. 
It follows that the two arguments of $C_{h,p,r}$ in \eqref{eq:Chp-smooth} are equal for $h=h_{p,r}^*$.
For smaller values of $h$ we have $C_{h,p,r}=(h/\pi)^r$, and then \eqref{ineq:smoothL2-1} coincides with \eqref{eq:Sande-intro}. Otherwise, for larger values of $h$, \eqref{ineq:smoothL2-1} coincides with the estimate for global polynomial approximation in Lemma~\ref{lem:poly}. 
Assuming a uniform knot sequence, we observe that the latter only holds for a rather small number of knot intervals $\nknots=(b-a)/h$ with respect to $p$.
For instance, if $p=10$ and $r=11$, then $h_{p,r}^*/(b-a)\approx0.17$ and so $N$ must be less than or equal to $5$ for the estimate in \eqref{ineq:smoothL2-1} to coincide with the estimate for global polynomial approximation. Similarly, one can check that if $p=10$ and $r=1$, then $N$ must be less than or equal to $7$.
\end{example}

It is easy to see that for fixed $p$ and small enough $h$, both estimates in Corollary~\ref{cor:smoothL2} coincide. Moreover, for fixed $h$ and large enough $p$, \eqref{ineq:smoothL2-1} is a sharper estimate than \eqref{ineq:smoothL2-2}.
On the other hand, as we illustrate in the next example, there are certain choices of $h$ and $p$ such that \eqref{ineq:smoothL2-2} is sharper than \eqref{ineq:smoothL2-1}.
\begin{example}
Let $r=2$ and $b-a=2$. Then, assuming
\begin{equation} \label{eq:assumption-h-p}
  \frac{\pi}{\sqrt{(p+1)(p+2)}}<h<\frac{\pi}{\sqrt{p(p+3)}},
\end{equation}
we have
\begin{align*}
  C_{h,p,1}C_{h,p-1,1}
  =\frac{h}{\pi}\frac{1}{\sqrt{(p+1)(p+2)}}
  <\min\left\{\left(\frac{h}{\pi}\right)^2,\frac{1}{\sqrt{p(p+1)(p+2)(p+3)}}\right\}
  =C_{h,p,2}.
\end{align*}
As a consequence, when $h$ satisfies \eqref{eq:assumption-h-p}, the error estimate \eqref{ineq:smoothL2-2} is sharper than \eqref{ineq:smoothL2-1}.
\end{example}

The estimates in Corollary~\ref{cor:smoothL2} hint towards a complex interplay between $h$ and $p$ in the sense that for a strongly refined grid (very small $h$), increasing the degree $p$ might give little or no benefit, and vice versa.

\begin{remark}\label{rmk:smoothL2-harmonic}
Using the Stirling formula (in the same way as in Remark~\ref{rmk:low-order-simple}), we have
\begin{equation*}
\sqrt{\frac{(p+1-r)!}{(p+1+r)!}}\leq %\left(\frac{e}{2(p+1)}\right)^{r^2/(p+1)}\leq 
\left(\frac{e}{2(p+1)}\right)^{r}.
\end{equation*}
Thus,
\begin{equation*}%\label{eq:Chpr-bound}
C_{h,p,r}\leq\min\left\{\left(\frac{h}{\pi}\right)^r,\left(\frac{e(b-a)}{4(p+1)}\right)^r\right\} = \left(\min\left\{\frac{h}{\pi},\frac{e(b-a)}{4(p+1)}\right\}\right)^r,
\end{equation*}
and by taking the harmonic mean of the two quantities in the above bound, we get
\begin{equation}\label{eq:harmonic}
\|u-S_pu\| \leq \left(\frac{2eh(b-a)}{e\pi(b-a)+4h(p+1)}\right)^r\|\partial^ru\|,
\end{equation}
for all $p\geq r-1$. Even though this estimate is less sharp than the result in Corollary~\ref{cor:smoothL2}, it has the benefit of always decreasing as the grid is refined and/or as the degree is increased. 
\end{remark}

\begin{remark}\label{rmk:smoothL2-harmonic-smallr}
For small values of $r$ (compared to $p$) we can improve upon the estimate in Remark~\ref{rmk:smoothL2-harmonic} as follows. 
Since $C_{h,p-i+1,1} \leq C_{h,p-r+1,1}$ for $i=1,\ldots,r$, Corollary~\ref{cor:smoothL2} implies that
\begin{equation*}
 \|u-S_pu\| \leq (C_{h,p-r+1,1})^r\|\partial^ru\|,
\end{equation*}
for all $p\geq r-1$. 
By taking the harmonic mean of the two quantities in the bound
\begin{equation*}%\label{eq:Chp}
C_{h,p-r+1,1}\leq \min\left\{\frac{h}{\pi},\frac{b-a}{2(p-r+2)}\right\},
\end{equation*}
we obtain
\begin{equation*}
\|u-S_pu\| \leq \left(\frac{2h(b-a)}{\pi(b-a)+2h(p-r+2)}\right)^r\|\partial^ru\|,
\end{equation*}
for all $p\geq r-1$.
This estimate is sharper than \eqref{eq:harmonic} if $p>\frac{e}{e-2}(r+\frac{2}{e}-2)$. Note that this is always the case if $p\geq 4(r-1)$.
\end{remark}

We now look at some error estimates for the Ritz projection.
Using inequality~\eqref{ineq:Ritzsimple} and Lemma~\ref{lem:simple} we obtain the following result.
\begin{corollary}\label{cor:smoothRitz}
Let $u\in H^r(a,b)$ be given.
For any $q=0,\ldots,r$ and knot sequence $\knots$, let $R_p^q$ be the projector onto $S_{p,\knots}=\mathcal{Z}_p$ defined in \eqref{def:Ritz}. Then, for any $\ell=0,\ldots,q$ we have
\begin{align*}
\|\partial^{\ell}(u-R^q_pu)\| &\leq C_{h,p-q,q-\ell} C_{h,p-q,r-q} \|\partial^ru\|,\\
\|\partial^{\ell}(u-R^q_pu)\| &\leq\left(C_{h,p-q,1}\cdots C_{h,p-2q+\ell + 1,1}\right)\left(C_{h,p-q,1}\cdots C_{h,p-r+1,1} \right) \|\partial^ru\|,
\end{align*}
for all $p\geq \max\{q,r-1,2q-\ell-1\}$.
\end{corollary}

\begin{remark}\label{rmk:smoothRitz-simple}
As a generalization of \cite[Theorem~3.1]{Sande:2019}, it follows from Corollary~\ref{cor:smoothRitz} that
for any $q=0,\ldots,r$ and $\ell=0,\ldots,q$,
\begin{equation}
\|\partial^{\ell}(u-R^q_pu)\| \leq \left(\frac{h}{\pi}\right)^{r-\ell}\|\partial^ru\|, \label{eq:smoothRitz-simple}
\end{equation}
for all $p\geq \max\{q,r-1,2q-\ell-1\}$.
Not only is this a very simple and explicit estimate, but it is also very useful for $h$ refinement.
Note that the error estimate for periodic splines in \cite[Theorem~4.1]{Sande:2019} is of the same form as \eqref{eq:smoothRitz-simple} for the corresponding Ritz projection in the case of periodic boundary conditions.
\end{remark}

\begin{remark}\label{rmk:smoothRitz-harmonic}
Following a similar argument as in Remark~\ref{rmk:smoothL2-harmonic}, we get
for any $q=0,\ldots,r$ and $\ell=0,\ldots,q$,
\begin{equation*}
\|\partial^{\ell}(u-R^q_pu)\| \leq  \left(\frac{2eh(b-a)}{e\pi(b-a)+4h(p-q+1)}\right)^{r-\ell} \|\partial^ru\|, 
\end{equation*}
for all $p\geq \max\{q,r-1,2q-\ell-1\}$.
In addition, following a similar argument as in Remark~\ref{rmk:smoothL2-harmonic-smallr}, we get
for any $q=0,\ldots,r$ and $\ell=0,\ldots,q$,
\begin{equation*}
\|\partial^{\ell}(u-R^q_pu)\| \leq  \left(\frac{2h(b-a)}{\pi(b-a)+2h(p+2-\max\{2q-\ell,r\})}\right)^{r-\ell} \|\partial^ru\|, 
\end{equation*}
for all $p\geq \max\{q,r-1,2q-\ell-1\}$.
The latter estimate is sharper than the former one if $p>\frac{e}{e-2}(\max\{2q-\ell,r\}+\frac{2(1-q)}{e}-2)$.
Even though these two estimates are less sharp than the result in Corollary~\ref{cor:smoothRitz}, they have the benefit of always decreasing as the grid is refined and/or as the degree is increased. 
They are therefore useful estimates for $h$-$p$-$k$ refinement.
\end{remark}

%%%%%%%%%%%%%
\section{Reduced spline spaces}\label{sec:reduced}
%%%%%%%%%%%%%
The goal of this section is to prove error estimates for the Ritz projection onto certain reduced spline spaces of maximal smoothness studied in \cite{Takacs:2016,Floater:2017,Hofreither:2017,Sogn:2018,Floater:2018,Sande:2019}. To do that we first prove a general result for any integral operator $K$ using ideas from \cite{Floater:2017,Floater:2018}. 

\subsection{General error estimates}\label{subsec:gen}

Let $K$ be any integral operator as in \eqref{eq:K}, and let $\mathcal{X}_0$ and $\mathcal{Y}_0$ be any finite dimensional subspaces of $L^2(a,b)$. We then define the subspaces $\mathcal{X}_\prec$ and $\mathcal{Y}_\prec$ in an analogous way to \eqref{eq:Xsimpl}, by
\begin{equation}\label{eq:Xs}
\mathcal{X}_\prec:=K(\mathcal{Y}_{\prec-1}), \quad \mathcal{Y}_\prec:=K^*(\mathcal{X}_{\prec-1}), 
\end{equation}
for $\prec\geq 1$. Finally, for any $\prec\geq0$, let $X_\prec$ be the $L^2$-projector onto $\mathcal{X}_\prec$ and $Y_\prec$ be the $L^2$-projector onto $\mathcal{Y}_\prec$.
\begin{lemma}\label{lem:gen}
For any $\prec\geq 1$ we have
\begin{align*}%\label{ineq:X}
\|K - KY_\prec\| \leq \|K^*-K^*X_{\prec-1}\| \leq \begin{cases}
	\|K-X_{0}K\|, &\prec\text{ odd},
	\\
	\|K^*-Y_0K^*\|, &\prec\text{ even}.
\end{cases}
\end{align*}
\end{lemma}
\begin{proof}
First, note that
 \begin{align*}
\|K - KY_\prec\| = \|K^*-Y_\prec K^*\| = \sup_{\|f\|\leq 1} \|K^*f-Y_\prec K^*f\|.
 \end{align*}
Next, observe that $K^*X_{\prec-1}$ maps into $\mathcal{Y}_\prec$ and since $Y_\prec K^*f$ is the best approximation of $K^*f$ in $\mathcal{Y}_\prec$ we must have
\begin{align*}
\sup_{\|f\|\leq 1} \|K^*f-Y_\prec K^*f\| \leq \sup_{\|f\|\leq 1} \|K^*f-K^*X_{\prec-1}f\| =  \|K^*-K^*X_{\prec-1}\|.
\end{align*}
This shows that $\|K - KY_\prec\|\leq \|K^*-K^*X_{\prec-1}\|$. Similarly, by swapping the roles of $K$ and $K^*$ we have $\|K^* - K^*X_{\prec}\|\leq \|K-KY_{\prec-1}\|$. The result then follows from induction on~$\prec$.
\end{proof}

\subsection{Error estimates for reduced spline spaces}\label{subsec:red}
In \cite{Takacs:2016,Floater:2017,Floater:2018,Sande:2019} error estimates for certain reduced spline spaces were shown. Here we prove a generalization of these results for the Ritz projections. Specifically, in \cite{Floater:2018} and \cite{Sande:2019} the spaces $\mathcal{S}_{p,\knots,0}$ and $\mathcal{S}_{p,\knots,1}$, defined by 
\begin{equation}\label{eq:Spi}
\begin{aligned}
\mathcal{S}_{p,\knots,0} &:= \{s\in \mathcal{S}_{p,\knots} :\, \partial^\alpha s(a)=\partial^\alpha s(b)=0,\ \ 0\leq \alpha\leq p,\ \ \alpha \text{ even}\}, \\
 \mathcal{S}_{p,\knots,1} &:= \{s\in \mathcal{S}_{p,\knots} :\, \partial^\alpha s(a)=\partial^\alpha s(b)=0,\ \ 0\leq \alpha\leq p,\ \ \alpha \text{ odd}\},
\end{aligned}
\end{equation}
were studied.
We further define the related spaces $\overline{\mathcal{S}}_{p,\knots,0}$ and $\overline{\mathcal{S}}_{p,\knots,1}$ by
\begin{equation}\label{eq:Spi-tilde}
\begin{aligned}
\overline{\mathcal{S}}_{p,\knots,0} &:= \{s\in \mathcal{S}_{p,\knots} :\, \partial^\alpha s(a)=\partial^\alpha s(b)=0,\ \ 0\leq \alpha< p,\ \ \alpha \text{ even}\},\\
 \overline{\mathcal{S}}_{p,\knots,1} &:= 
\{s\in \mathcal{S}_{p,\knots} :\, \partial^\alpha s(a)=\partial^\alpha s(b)=0,\ \ 0\leq \alpha< p,\ \ \alpha \text{ odd}\}.
\end{aligned}
\end{equation}
For uniform knot sequences, the spaces $\overline{\mathcal{S}}_{p,\knots,1}$ are exactly the reduced spline spaces investigated in \cite[Definition~5.1]{Takacs:2016}.
Observe that $\mathcal{S}_{p,\knots,0} \subseteq \overline{\mathcal{S}}_{p,\knots,0}$ where equality holds for $p$ odd and $\mathcal{S}_{p,\knots,1} \subseteq \overline{\mathcal{S}}_{p,\knots,1}$ where equality holds for $p$ even. Observe further that in the case $p=0$ all the spaces in \eqref{eq:Spi} and \eqref{eq:Spi-tilde} equal the standard spline space $\mathcal{S}_{0,\knots}$ \emph{except} for $\mathcal{S}_{0,\knots,0}$.

For a specific (degree-dependent) knot sequence $\knots$ it was shown in \cite{Floater:2018} that the spline spaces in \eqref{eq:Spi} are optimal for certain $n$-width problems. Later it was shown in \cite{Sande:2019} that if $n$ is the dimension of these optimal spaces, then they converge to the space spanned by the $n$ first eigenfunctions of the Laplacian (with either Dirichlet or Neumann boundary conditions) as their degree $p$ increases. 
Convergence in the case of periodic boundary conditions was also studied in \cite{Sande:2019}.

Staying consistent with the notation in \cite{Pinkus:85,Floater:2017,Floater:2018} we define the integral operator $K_1$ by
\begin{equation*}
K_1:=(I-P_0)K,
\end{equation*}
where $P_0$ denotes the $L^2$-projector onto $\mathcal{P}_0$, and $K$ is the integral operator in \eqref{eq:Kint}. One can verify that if $u=K_1f$ then $\partial u=f$ and $u\perp 1$. Moreover, since $K_1^*=K^*(I-P_0)$ it follows that if $u=K_1^*f$ then $\partial u=(P_0-I)f$ and $u(a)=u(b)=0$. Using these properties it was shown in \cite{Floater:2018} that
\begin{equation}\label{eq:Spi-orth}
\begin{aligned}
\mathcal{S}_{p,\knots,0}&=K_1^*(\mathcal{S}_{p-1,\knots,1}),
\\
\mathcal{S}_{p,\knots,1} &= \mathcal{P}_0\oplus K_1(\mathcal{S}_{p-1,\knots,0}),
\end{aligned}
\end{equation}
for all $p\geq 1$, since the derivative of a spline is a spline  of one degree lower on the same knot sequence. For the spline spaces in \eqref{eq:Spi-tilde} we deduce by the same argument that
\begin{equation}\label{eq:Spi-tilde-orth}
\begin{aligned}
\overline{\mathcal{S}}_{p,\knots,0}&=K_1^*(\overline{\mathcal{S}}_{p-1,\knots,1}),
\\
\overline{\mathcal{S}}_{p,\knots,1} &= \mathcal{P}_0\oplus K_1(\overline{\mathcal{S}}_{p-1,\knots,0}),
\end{aligned}
\end{equation}
for all $p\geq 1$.

Let  $S_{p,i}:L^2(a,b)\to  \mathcal{S}_{p,\knots,i}$, $i=0,1$, denote the $L^2$-projector.
Analogously to \eqref{def:Ritz} we define, for $p\geq 1$, the Ritz projector $R_{p,0}:H^1_0(a,b)\to  \mathcal{S}_{p,\knots,0}$ by
\begin{equation}\label{def:RitzS0}
(\partial R_{p,0} u,\partial v) = (\partial  u, \partial v), \quad \forall v\in\mathcal{S}_{p,\knots,0},
\end{equation}
and the Ritz projector $R_{p,1}:H^1(a,b)\to  \mathcal{S}_{p,\knots,1}$ by
\begin{equation}\label{def:RitzS1}
\begin{aligned}
(\partial R_{p,1} u,\partial v) &= (\partial  u, \partial v), \quad \forall v\in\mathcal{S}_{p,\knots,1},
\\
(R_{p,1} u,1)&=(u,1).
\end{aligned}
\end{equation}
Using the above definitions, together with \eqref{eq:Spi-orth}, we find that $R_{p,0}=K_1^*S_{p-1,1}$ and  $R_{p,1}=P_0+K_1S_{p-1,0}$.

Lastly, we define the quantity $\widehat{h}$ by
\begin{align*}
\widehat h:=\max\{2(\xi_{1}-\xi_0),(\xi_{2}-\xi_1),(\xi_{3}-\xi_2),\ldots, (\xi_{\nknots}-\xi_{\nknots-1}),2(\xi_{\nknots+1}-\xi_{\nknots})\},
\end{align*}
where we observe that the difference between $\widehat h$ and the $h$ in \eqref{eq:hmax} is that the length of the first and the last knot interval is scaled by a factor of $2$.
To prove the error estimates for our Ritz projections onto the sequences of spaces in \eqref{eq:Spi} we make use of the next lemma.
\begin{lemma}\label{lem:S0}
For any $u\in H^1(a,b)$ we have
\begin{equation*}
\|u-S_{0,1}u\|\leq \frac{h}{\pi}\|\partial u\|,
\end{equation*}
and for any $v\in H^1_0(a,b)$ we have
\begin{equation*}
\|v-S_{0,0}v\|\leq \frac{\widehat h}{\pi}\|\partial v\|.
\end{equation*}
\end{lemma}
\begin{proof}
These results follow from the Poincar\'e inequality. We refer the reader to \cite[Theorem~1.1 and Lemma~8.1]{Sande:2019} for the details.
\end{proof}
Using the above lemma together with Lemma~\ref{lem:gen} we obtain the desired error estimates.
\begin{theorem}\label{thm:Spi}
Let $p\geq 0$ be given. Then, for any $u\in H^1(a,b)$ we have
\begin{equation*}
\begin{aligned}
\|u-R_{p,1}u\|&\leq \frac{\widehat h}{\pi}\|\partial u\|, &&\pnew \text{ odd},
\\
\|u-R_{p,1}u\|&\leq \frac{h}{\pi}\|\partial u\|, &&\pnew \text{ even},
\end{aligned}
\end{equation*}
and for any $v\in H^1_0(a,b)$ we have
\begin{equation*}
\begin{aligned}
\|v-R_{p,0}v\|&\leq \frac{h}{\pi}\|\partial v\|, &&\pnew \text{ odd},
\\
\|v-R_{p,0}v\|&\leq \frac{\widehat h}{\pi}\|\partial v\|, &&\pnew \text{ even}.
\end{aligned}
\end{equation*}
\end{theorem}
\begin{proof}
Define the spaces $\mathcal{S}_{p}$ by
\begin{equation*}
\mathcal{S}_{p}:=\{s\in  \mathcal{S}_{p,\knots,1} : s\perp 1\}.
\end{equation*}
Using \eqref{eq:Spi} we find that if $K_1$ plays the role of the generic integral operator $K$ in Section~\ref{subsec:gen}, then the spaces $\mathcal{S}_{p}$ are examples of the $\mathcal{X}_\prec$ in \eqref{eq:Xs} and the spaces $\mathcal{S}_{p,\knots,0}$ are examples of the  $\mathcal{Y}_\prec$ in \eqref{eq:Xs} for $p=\prec$. 

Moreover, using the definition of $K_1$ we observe that $H^1(a,b)=\mathcal{P}_0\oplus K_1(L^2(a,b))$. Thus, any function $u\in H^1(a,b)$ can be decomposed as $u=c+K_1f$ for $c\in\mathcal{P}_0$ and $f\in L^2(a,b)$. Using Lemma~\ref{lem:S0} we then find that 
\begin{equation*}
\|Kf-S_{0,1}Kf\|=\|u-S_{0,1}u\|\leq \frac{h}{\pi}\|\partial u\|= \frac{h}{\pi}\|f\|,
\end{equation*}
since $\mathcal{P}_0\subset  \mathcal{S}_{p,\knots,1}$, and so $\|K-S_{0,1}K\|\leq h/\pi$.
Furthermore, it was shown in \cite{Floater:2018} that $H^1_0(a,b)=K_1^*(L^2(a,b))$ and so any function $v\in H^1_0(a,b)$ can be written as $v=K_1^*g$ for $g\in L^2(a,b)$. Again, using Lemma~\ref{lem:S0}, we find that 
\begin{equation*}
\|K^*g-S_{0,0}K^*g\|=\|v-S_{0,0}v\|\leq \frac{\widehat h}{\pi}\|\partial v\|= \frac{\widehat h}{\pi}\|g\|,
\end{equation*}
and $\|K^*-S_{0,0}K^*\|\leq \widehat{h}/\pi$. The result then follows from Lemma~\ref{lem:gen} since $R_{p,0}=K_1^*S_{p-1,1}$ and  $R_{p,1}=P_0+K_1S_{p-1,0}$.
\end{proof}

Let $\overline{S}_{p,i}:L^2(a,b)\to \overline{\mathcal{S}}_{p,\knots,i}$, $i=0,1$, denote the $L^2$-projector.
We then define the Ritz projector $\overline{R}_{p,0}:H^1(a,b)\to \overline{\mathcal{S}}_{p,\knots,0}$ in a completely analogous way to \eqref{def:RitzS0} and $\overline{R}_{p,1}:H^1(a,b)\to \overline{\mathcal{S}}_{p,\knots,1}$ in a completely analogous way to \eqref{def:RitzS1}.
As before, using \eqref{eq:Spi-tilde-orth} we find that $\overline{R}_{p,0}=K_1^*\overline{S}_{p-1,1}$ and  $\overline{R}_{p,1}=P_0+K_1\overline{S}_{p-1,0}$.

\begin{theorem}\label{thm:Spi-tilde}
Let $p\geq0$ be given. Then, for any $u\in H^1(a,b)$ we have
\begin{equation*}
\|u-\overline{R}_{p,1}u\|\leq \frac{h}{\pi}\|\partial u\|,
\end{equation*}
and for any $v\in H^1_0(a,b)$ we have
\begin{equation*}
\|v-\overline{R}_{p,0}v\|\leq \frac{h}{\pi}\|\partial v\|,
\end{equation*}
\end{theorem}
\begin{proof}
This result follows from a similar argument as in the proof of Theorem~\ref{thm:Spi}. The main change being that in the case $p=0$ we have
$
\overline{\mathcal{S}}_{0,\knots,0}=\mathcal{S}_{0,\knots},
$
and so $\|K^*-\overline{S}_{0,0}K^*\|\leq {h}/\pi$.
\end{proof}

\begin{remark}\label{rmk:inverse}
The reduced spline spaces defined in \eqref{eq:Spi} and \eqref{eq:Spi-tilde} all satisfy the boundary conditions stated in \cite[Theorem~9.1]{Sande:2019}. Hence, any element $s$ in such spaces satisfies the following inverse inequality:
\begin{equation*}
\|s'\|\leq \frac{2\sqrt{3}}{\hmin}\|s\|, 
\end{equation*}
where $\hmin$ is the minimum knot distance.
\end{remark}

\begin{remark}
As the error estimates in Theorems~\ref{thm:Spi} and~\ref{thm:Spi-tilde} are complemented with the inverse inequality in Remark~\ref{rmk:inverse}, the reduced spline spaces defined in \eqref{eq:Spi} and \eqref{eq:Spi-tilde} can be used to design fast iterative (multigrid) solvers for linear systems arising from spline discretization methods~\cite{Hofreither:2017,Sogn:2018}.
\end{remark}

%%%%%%%%%%%%%
\section{Tensor-product spline spaces}\label{sec:tensor}
%%%%%%%%%%%%%
In this section we describe how to extend our error estimates to the case of tensor-product spline spaces. 
We start by introducing some notation. 
Consider the $d$-dimensional domain $\Ddom:=(a_1,b_1)\times(a_2,b_2)\times \cdots \times (a_d,b_d)$, and let $\|\cdot\|_{\Ddom}$ denote the $L^2(\Ddom)$-norm.
Moreover, we deal with the standard Sobolev spaces on $\Ddom$ defined by
\begin{equation*}
H^r(\Ddom):= \{u\in L^2(\Ddom) : \partial_1^{\alpha_1}\cdots\partial_d^{\alpha_d} u \in L^2(\Ddom),\, 1\leq\alpha_1+\cdots+\alpha_d\leq r \}.
\end{equation*}

For $i=1,\ldots,d$, let $\mathcal{Z}_{\prec_i}$ be a finite dimensional subspace of $L^2(a_i,b_i)$ as in \eqref{eq:Xsimpl} with $K$ as in \eqref{eq:Kint}, and define the tensor-product space $\mathcal{Z}_{\bf \prec}:=\mathcal{Z}_{\prec_1}\otimes \mathcal{Z}_{\prec_2}\otimes\cdots\otimes \mathcal{Z}_{\prec_d}$.
We only investigate projectors onto $\mathcal{Z}_{\bf \prec}$ of the form $\Pi:=\Pi_1\otimes\Pi_2\otimes\cdots\otimes\Pi_d$. 
To simplify notation, we use the following convention: when applying the univariate operator $\Pi_i$ to a $d$-variate function $u$, we mean that $\Pi_i$ acts on the $i$-th variable of $u$ while the others are considered as parameters. In this perspective, we have
$\Pi=\Pi_1\circ\Pi_2\circ\cdots\circ\Pi_d$.

%\subsection{Tensor-product $L^2$-projection}
We first study error estimates for the $L^2(\Ddom)$-projection onto $\mathcal{Z}_{\bf \prec}$, denoted by 
$Z_{\bf \prec}:=Z_{\prec_1}\otimes Z_{\prec_2}\otimes\cdots\otimes Z_{\prec_d}$.
The following result can be concluded from the univariate error estimates using a standard argument (see, e.g., \cite{Schwab:99,Buffa:11,Takacs:2016,Sande:2019,Bressan:2019}), but for the sake of completeness we include the argument here.
\begin{theorem}\label{thm:tensorL2}
For any $u\in L^2(\Ddom)$ we have
\begin{align*}
\|u-Z_{\bf \prec}u\|_{\Ddom}\leq \sum_{i=1}^d\|u-Z_{\prec_i}u\|_{\Ddom},
\end{align*}
and consequently, if $u\in H^r(\Ddom)$ and $\mathcal{P}_{r-1}\subseteq\mathcal{Z}_{\prec_i}$ for all $i=1,\ldots,d$, we have
\begin{align}\label{ineq:tensorL2}
\|u-Z_{\bf \prec}u\|_{\Ddom} \leq \sum_{i=1}^d \Cgen_{\prec_i,r}\|\partial_i^ru\|_{\Ddom}.
\end{align}
\end{theorem}
\begin{proof}
We only consider the case $d=2$. The generalization to arbitrary $d$ is straightforward.
From the triangle inequality we obtain
\begin{align*}
\|u-Z_{\prec_1}\otimes Z_{\prec_2}u\|_{\Ddom}&\leq \|u-Z_{\prec_1}u\|_{\Ddom}+\|Z_{\prec_1}u-Z_{\prec_1}\otimes Z_{\prec_2}u\|_{\Ddom}
\\
&\leq \|u-Z_{\prec_1}u\|_{\Ddom}+\|Z_{\prec_1}\|\,\|u- Z_{\prec_2}u\|_{\Ddom}
\\
&\leq \|u-Z_{\prec_1}u\|_{\Ddom}+\|u-Z_{\prec_2}u\|_{\Ddom},
\end{align*}
since the $L^2(\Ddom)$-operator norm of $Z_{\prec_1}$ is equal to $1$. Applying Theorem~\ref{thm:L2} we then obtain \eqref{ineq:tensorL2}.
\end{proof}
\begin{remark}\label{rem:tensorL2}
For simplicity let $\Ddom=(0,1)^d$.
Note that \eqref{ineq:tensorL2} actually holds for all functions $u$ in the larger Sobolev space
\begin{equation*}
\bigcap_{i=1}^d L^2(0,1)^{i-1}\otimes H^{r}(0,1)\otimes L^2(0,1)^{d-i} \supseteq H^r(\Ddom).
\end{equation*}
We make use of a similar Sobolev space in Section~\ref{subsec:bent}.
\end{remark}

For tensor-product spline spaces of arbitrary smoothness, let $S^{\bf k}_{\bf p}:=S^{k_1}_{p_1}\otimes\cdots\otimes S^{k_d}_{p_d}$ denote the $L^2(\Ddom)$-projector onto 
$\mathcal{S}^{\bf k}_{{\bf p},\dknots} :=\mathcal{S}^{k_1}_{p_1,\knots_1}\otimes\cdots\otimes \mathcal{S}^{k_d}_{p_d,\knots_d}$.
For maximally smooth spline spaces, let
$S_{\bf p}:=S_{p_1}\otimes\cdots\otimes S_{p_d}$ denote the $L^2(\Ddom)$-projector onto 
$\mathcal{S}_{{\bf p},\dknots} :=\mathcal{S}_{p_1,\knots_1}\otimes\cdots\otimes \mathcal{S}_{p_d,\knots_d}$.
Error estimates for these spaces can be immediately obtained by replacing $\Cgen_{\prec_i,r}$ in Theorem~\ref{thm:tensorL2} with the constants derived in Corollaries~\ref{cor:low-order} and~\ref{cor:smoothL2}. Let $h_i$ denote the maximal knot distance in $\knots_i$ for $i=1,\ldots,d$.
\begin{corollary}\label{cor:tensorL2}
For any $u\in H^r(\Ddom)$ we have
\begin{align*}
\|u-S^{\bf k}_{\bf p}u\|_{\Ddom} &\leq \sum_{i=1}^d C_{h_i,p_i,k_i,r}\|\partial_i^ru\|_{\Ddom},
\end{align*}
and
\begin{align*}
\|u-S_{\bf p}u\|_{\Ddom} &\leq \sum_{i=1}^d C_{h_i,p_i,r}\|\partial_i^ru\|_{\Ddom},
\\
\|u-S_{\bf p}u\|_{\Ddom} &\leq \sum_{i=1}^d C_{h_i,p_i,1}C_{h_i,p_i-1,1}\cdots C_{h_i,p_i-r+1,1}\|\partial_i^ru\|_{\Ddom},
\end{align*}
for all $p_i\geq r-1$.
\end{corollary}
\begin{example}
Let $h:=\max\{h_1,h_2,\ldots,h_d\}$. Then, for any $u\in H^r(\Ddom)$ we have
\begin{equation*}
\|u-S_{\bf p}u\|_{\Ddom}\leq \sum_{i=1}^d \left(\frac{h_i}{\pi}\right)^{r}\|\partial_i^ru\|_{\Ddom} \leq \left(\frac{h}{\pi}\right)^{r}\sum_{i=1}^d \|\partial_i^ru\|_{\Ddom},
\end{equation*}
for all $p_i\geq r-1$. 
\end{example}

%\subsection{Tensor-product Ritz projection}
Let us now focus on error estimates for tensor products of the Ritz projection in \eqref{def:Ritz}. For simplicity of notation, we only consider the case $q=1$ and $d=2$.
Define the tensor-product Ritz projector $R_{\bf \prec}:H^1(a_1,b_1)\otimes H^1(a_2,b_2)\to \mathcal{Z}_{\prec_1}\otimes \mathcal{Z}_{\prec_2}$ by 
\begin{equation*}
R_{\bf \prec}:=R^1_{\prec_1}\otimes R^1_{\prec_2}.
\end{equation*}
Note that $H^1(a_1,b_1)\otimes H^1(a_2,b_2)$ consists of functions $u\in L^2(\Ddom)$ such that $\partial_1u\in L^2(\Ddom)$, $\partial_2u\in L^2(\Ddom)$ and $\partial_{1}\partial_2u\in L^2(\Ddom)$. We thus have $H^2(\Ddom)\subset H^1(a_1,b_1)\otimes H^1(a_2,b_2)\subset H^1(\Ddom)$.

\begin{lemma}\label{lem:tensorRitz}
Let $u\in H^1(a_1,b_1)\otimes H^1(a_2,b_2)$ be given. Then, for all $\prec_1,\prec_2\geq 1$ we have
\begin{align*}
\|u-R_{\bf \prec}u\|_{\Ddom} &\leq\|u-R_{\prec_1}^1u\|_{\Ddom}+\|u-R_{\prec_2}^1u\|_{\Ddom}
\\
&\quad +\min\{\Cgen_{\prec_1-1,1}\|\partial_{1}u-R_{\prec_2}^1\partial_1u\|_{\Ddom},\, \Cgen_{\prec_2-1,1}\|\partial_{2}u-R_{\prec_1}^1\partial_2u\|_{\Ddom}\},
\\
\|\partial_1(u-R_{\bf \prec}u)\|_{\Ddom} &\leq\|\partial_1(u-R_{\prec_1}^1u)\|_{\Ddom}+\|\partial_{1}u-R_{\prec_2}^1\partial_1u\|_{\Ddom},
\\
\|\partial_1\partial_2(u-R_{\bf \prec}u)\|_{\Ddom} &\leq\|\partial_1\partial_2u-Z_{\prec_1-1}\partial_1\partial_2u\|_{\Ddom}+\|\partial_{1}\partial_2u-Z_{\prec_2-1}\partial_1\partial_2u\|_{\Ddom}.
\end{align*}
\end{lemma}
\begin{proof}
From \eqref{ineq:stab:b} and by adding and subtracting $R_{\prec_1}^1u$ we obtain
\begin{align*}
\|u-R_{\bf \prec}u\|_{\Ddom} &\leq \|u-R_{\prec_1}^1u\|_{\Ddom} + \|R_{\prec_1}^1(u-R_{\prec_2}^1u)\|_{\Ddom}
\\
&\leq \|u-R_{\prec_1}^1u\|_{\Ddom} + \|u-R_{\prec_2}^1u\|_{\Ddom} + \Cgen_{\prec_1-1,1}\|\partial_1(u-R_{\prec_2}^1u)\|_{\Ddom},
\end{align*}
and similarly for $R_{\prec_2}^1u$. The first result now follows since $\partial_i$ commutes with $R_{\prec_j}^1$ for $i\neq j$.
Analogously, using \eqref{ineq:stab:a} we obtain
\begin{align*}
\|\partial_1(u-R_{\bf \prec}u)\|_{\Ddom} &\leq \|\partial_1(u-R_{\prec_1}^1u)\|_{\Ddom} + \|\partial_1R_{\prec_1}^1(u-R_{\prec_2}^1u)\|_{\Ddom}
\\
&\leq \|\partial_1(u-R_{\prec_1}^1u)\|_{\Ddom} + \|\partial_1(u-R_{\prec_2}^1u)\|_{\Ddom},
\end{align*}
and the second result follows. For the third result we use the commuting relation $\partial_iR_{\prec_i}^1=Z_{\prec_i-1}\partial_i$, $i=1,2$, to conclude that $\partial_1\partial_2R_{\bf \prec}=Z_{\bf \prec-1}\partial_1\partial_2$, and we apply Theorem~\ref{thm:tensorL2}.
\end{proof}

By using Theorem~\ref{thm:Ritz} we can now achieve error estimates for the tensor-product Ritz projection.
If the function $u$ is only assumed to be in $H^1(a_1,b_1)\otimes H^1(a_2,b_2)$ then one obtains the ``unbalanced'' estimate:
\begin{align*}
\|u-R_{\bf \prec}u\|_{\Ddom}\leq \Cgen_{\prec_1-1,1}\|\partial_1u\|_{\Ddom}+\Cgen_{\prec_2-1,1}\|\partial_2u\|_{\Ddom}+\Cgen_{\prec_1-1,1}\Cgen_{\prec_2-1,1}\|\partial_1\partial_2u\|_{\Ddom},
\end{align*}
for all $\prec_1,\prec_2\geq 1$. 
Indeed, the partial derivatives involved in this estimate are not of the same order.
This can be resolved by requiring higher Sobolev smoothness.
If $u\in H^2(\Ddom)$ then, for all $\prec_1,\prec_2\geq 1$ we have
\begin{align*}
\|u-R_{\bf \prec}u\|_{\Ddom}\leq (\Cgen_{\prec_1-1,1})^2\|\partial_1^2u\|_{\Ddom}
+(\Cgen_{\prec_2-1,1})^2\|\partial_2^2u\|_{\Ddom}+\Cgen_{\prec_1-1,1}\Cgen_{\prec_2-1,1}\|\partial_1\partial_2u\|_{\Ddom}.
\end{align*}
This is a special case of the following more general statement.
\begin{theorem}\label{thm:tensorRitz}
 Let $u\in H^r(\Ddom)$ for $r\geq2$ be given. If $\mathcal{P}_{r-2}\subseteq \mathcal{Z}_{\prec_1-1}\cap\mathcal{Z}_{\prec_2-1}$ for $\prec_1,\prec_2\geq1$, then
\begin{align*}
\|u-R_{\bf \prec}u\|_{\Ddom}&\leq \Cgen_{\prec_1-1,1}\Cgen_{\prec_1-1,r-1}\|\partial_1^ru\|_{\Ddom}+\Cgen_{\prec_2-1,1}\Cgen_{\prec_2-1,r-1}\|\partial_2^ru\|_{\Ddom} 
\\
&\quad +\Cgen_{\prec_1-1,1}\Cgen_{\prec_2-1,1}\min\left\{\Cgen_{\prec_2-1,r-2}\|\partial_1\partial_2^{r-1}u\|_{\Ddom}, \,\Cgen_{\prec_1-1,r-2}\|\partial_1^{r-1}\partial_2u\|_{\Ddom}\right\},
\end{align*}
and
\begin{align*}
\|\partial_1(u-R_{\bf \prec}u)\|_{\Ddom}&\leq \Cgen_{\prec_1-1,r-1}\|\partial_1^ru\|_{\Ddom}+\Cgen_{\prec_2-1,1}\Cgen_{\prec_2-1,r-2}\|\partial_1\partial_2^{r-1}u\|_{\Ddom}, 
\\
\|\partial_2(u-R_{\bf \prec}u)\|_{\Ddom}&\leq \Cgen_{\prec_1-1,1}\Cgen_{\prec_1-1,r-2}\|\partial_1^{r-1}\partial_2u\|_{\Ddom}+\Cgen_{\prec_2-1,r-1}\|\partial_2^ru\|_{\Ddom},
\\
\|\partial_1\partial_2(u-R_{\bf \prec}u)\|_{\Ddom}&\leq  \Cgen_{\prec_1-1,r-2}\|\partial_1^{r-1}\partial_2u\|_{\Ddom}+\Cgen_{\prec_2-1,r-2}\|\partial_1\partial_2^{r-1}u\|_{\Ddom}.
\end{align*}
\end{theorem}
\begin{proof}
Using Lemma \ref{lem:tensorRitz} and Theorem \ref{thm:Ritz} we find that
\begin{align*}
&\|u-R_{\bf \prec}u\|_{\Ddom} \\
&\quad\leq\|u-R_{\prec_1}^1u\|_{\Ddom}+\|u-R_{\prec_2}^1u\|_{\Ddom}
 +\min\{\Cgen_{\prec_1-1,1}\|\partial_{1}u-R_{\prec_2}^1\partial_1u\|_{\Ddom},\, \Cgen_{\prec_2-1,1}\|\partial_{2}u-R_{\prec_1}^1\partial_2u\|_{\Ddom}\}
\\
&\quad\leq \Cgen_{\prec_1-1,1}\Cgen_{\prec_1-1,r-1}\|\partial_1^ru\|_{\Ddom}+\Cgen_{\prec_2-1,1}\Cgen_{\prec_2-1,r-1}\|\partial_2^ru\|_{\Ddom} 
\\
&\quad\quad +\min\left\{\Cgen_{\prec_1-1,1}\Cgen_{\prec_2-1,1}\Cgen_{\prec_2-1,r-2}\|\partial_1\partial_2^{r-1}u\|_{\Ddom}, \,\Cgen_{\prec_2-1,1}\Cgen_{\prec_1-1,1}\Cgen_{\prec_1-1,r-2}\|\partial_1^{r-1}\partial_2u\|_{\Ddom}\right\},
\end{align*}
which proves the first result. The other results follow by a similar argument.
\end{proof}
In the spirit of Corollary~\ref{cor:tensorL2}, using results from Sections~\ref{sec:low} and \ref{sec:max}, the above theorem can be used to obtain error estimates for tensor-product Ritz projections onto spline spaces of any smoothness. We end this section with two examples.

\begin{example}
Let $R^{\bf k}_{\bf p}:=R^{1,k_1}_{p_1}\otimes R^{1,k_2}_{p_2}$ be the tensor-product Ritz projector onto $\mathcal{S}_{{\bf p},\dknots}^{\bf k}$, and let
$h:=\max\{h_1,h_2\}$ and $p-k:=\min\{p_1-k_1,p_2-k_2\}$. Then, for any $u\in H^r(\Ddom)$, $r\geq2$, we have
\begin{equation*}
\|u-R^{\bf k}_{\bf p}u\|_{\Ddom}\leq \left(\dfrac{e\,h}{4(p-k)}\right)^r\left(\|\partial_1^ru\|_{\Ddom}+\|\partial_2^ru\|_{\Ddom} + \|\partial_{12}^ru\|_{\Ddom}\right),
\end{equation*}
where we slightly abuse notation by letting
\begin{equation*}
\|\partial_{12}^ru\|_{\Ddom}:=\min\left\{\|\partial_1\partial_2^{r-1}u\|_{\Ddom},\|\partial_1^{r-1}\partial_2u\|_{\Ddom}\right\},
\end{equation*}
for all $p_1,p_2\geq r-1$.
\end{example}

\begin{example}\label{ex:tensorRitz-smooth}
Let $R_{\bf p}:=R^1_{p_1}\otimes R^1_{p_2}$ be the tensor-product Ritz projector onto $\mathcal{S}_{{\bf p},\dknots}$, and
let $h:=\max\{h_1,h_2\}$. Then, for any $u\in H^2(\Ddom)$ and for $0\leq\ell_1,\ell_2\leq1$ we have
\begin{equation*}
\|\partial_1^{\ell_1}\partial_2^{\ell_2}(u-R_{\bf p}u)\|_{\Ddom}\leq \left(\frac{h}{\pi}\right)^{2-\ell_1-\ell_2}\left(\|\partial_1^2u\|_{\Ddom}+\|\partial_2^2u\|_{\Ddom}+\|\partial_1\partial_2u\|_{\Ddom}\right),
\end{equation*}
for all $p_1,p_2\geq 1$.
In general, for any $u\in H^r(\Ddom)$, $r\geq2$, and for $0\leq\ell_1,\ell_2\leq1$ we have
\begin{equation*}
\|\partial_1^{\ell_1}\partial_2^{\ell_2}(u-R_{\bf p}u)\|_{\Ddom}\leq \left(\frac{h}{\pi}\right)^{r-\ell_1-\ell_2}\left(\|\partial_1^ru\|_{\Ddom} +\|\partial_2^ru\|_{\Ddom} + \|\partial_1\partial_2^{r-1}u\|_{\Ddom} + \|\partial_1^{r-1}\partial_2u\|_{\Ddom}\right),
\end{equation*}
for all $p_1,p_2\geq r-1$. 
\end{example}

Similar results hold for the tensor products of the reduced spline spaces in Section~\ref{subsec:red}. The results of this section can also be generalized to higher order Ritz projections in a straightforward way.

%%%%%%%%%%%%%
\section{Mapped geometry}\label{sec:map}
%%%%%%%%%%%%%
Motivated by IGA, in this section we consider error estimates for spline spaces defined on a mapped (single-patch) domain.
Let $\Ddom=(0,1)^d$ be the reference domain, $\Pdom$ the physical domain, and ${\bf G}:\Ddom\to\Pdom\subset\RR^d$ the geometric mapping defining $\Pdom$. We assume that the mapping ${\bf G}$ is a bi-Lipschitz homeomorphism.
As a general rule, we indicate quantities and operators that refer to the (mapped) physical domain by means of $\tilde{~}$. In particular, the derivative operator with respect to physical variables is denoted by $\dpartial$.

Define the space $\widetilde{\mathcal{Z}}_{\bf \prec}$ as the push-forward of the tensor-product space $\mathcal{Z}_{\bf \prec}$ with respect to the mapping ${\bf G}$. Specifically, let
\begin{equation}\label{eq:mappedSpace}
\widetilde{\mathcal{Z}}_{\bf \prec}:=\{s\circ {\bf G}^{-1}: s\in \mathcal{Z}_{\bf \prec}\}.
\end{equation}
Furthermore, for any projector $\Pi: L^2(\Ddom)\to\mathcal{Z}_{\bf\prec}$ we let $\widetilde{\Pi}: L^2(\Pdom)\to \widetilde{\mathcal{Z}}_{\bf\prec}$ denote the projector defined by
\begin{equation}\label{eq:mappedProj}
\widetilde{\Pi} \tilde{u} := (\Pi(\tilde{u}\circ {\bf G}))\circ {\bf G}^{-1},\quad\forall \tilde{u}\in L^2(\Pdom).
\end{equation}
Using a standard substitution argument we obtain the following result.
\begin{lemma}\label{lem:map-generic}
For $\tilde{u}\in L^2(\Pdom)$ and ${\bf G}\in (W^{1,\infty}(\Ddom))^d$ let $u:=\tilde{u}\circ {\bf G}\in L^2(\Ddom)$. Then, for any projector $\widetilde{\Pi}$ we have
\begin{equation*}
\|\tilde{u}-\widetilde{\Pi}\tilde{u}\|_{\Pdom}\leq \|\det\nabla {\bf G}\|_{L^\infty(\Ddom)}\|u- \Pi u\|_{\Ddom}.
\end{equation*}
\end{lemma}

\subsection{Smooth geometry}
If we, similar to \cite{Hofreither:2017,Takacs:2018}, make the assumption that the geometry map ${\bf G}$ is sufficiently globally smooth, then we can easily extend the results from Section~\ref{sec:tensor} using techniques from \cite{Bazilevs:2006,Buffa:14}. Specifically, in this subsection we assume ${\bf G}\in (W^{r,\infty}(\Ddom))^d$, which implies that $u:=\tilde{u}\circ {\bf G}\in H^r(\Ddom)$ whenever $\tilde{u}\in H^r(\Pdom)$.  We further assume ${\bf G}^{-1}\in (W^{1,\infty}(\Pdom))^d$.

We define the mapped $L^2$-projector $\widetilde{Z}_{\bf\prec}:L^2(\Pdom)\to\widetilde{\mathcal{Z}}_{\bf \prec}$ by taking $\Pi=Z_{\bf\prec}$ in \eqref{eq:mappedProj}. Then, combining Lemma~\ref{lem:map-generic} and Theorem~\ref{thm:tensorL2} gives rise to the following estimate.

\begin{lemma}\label{lem:smoothG}
Let ${\bf G}\in (W^{r,\infty}(\Ddom))^d$. Then, for any $\tilde{u}\in H^r(\Pdom)$ we have
\begin{equation*}
\|\tilde{u}-\widetilde{Z}_{\bf\prec}\tilde{u}\|_{\Pdom}\leq \|\det\nabla {\bf G}\|_{L^\infty(\Ddom)}\sum_{i=1}^d\Cgen_{\prec_i,r}\|\partial_i^r(\tilde{u}\circ {\bf G})\|_{\Ddom},
\end{equation*}
for all $\prec_i\geq r-1$.
\end{lemma}

Using %the multi-dimensional high-order chain rule 
a slightly simplified version of the multivariate Fa\`a di Bruno formula in \cite{Constantine:1996} and substituting back to the physical domain, we obtain an error estimate in a more classical form. To this end, we set ${\bf G}:=(G_1,\ldots,G_d)$ and define 
\begin{equation*}
C_{\bf G}:=\|\det\nabla {\bf G}\|_{L^\infty(\Ddom)}\|\det\widetilde\nabla {\bf G}^{-1}\|_{L^\infty(\Pdom)},
\end{equation*}
and
% \begin{equation*}
% C_{{\bf G},i,r,{\bf j}}:=\biggl\|
% \sum_{I(r,{\bf j})} r! \prod_{m=1}^r 
% \frac{\bigl(\partial_i^{\ell_{m}}G_1\bigr)^{k_{m,1}}\cdots \bigl(\partial_i^{\ell_{m}}G_d\bigr)^{k_{m,d}}}
% {\bigl(k_{m,1}!\cdots k_{m,d}!\bigr) \bigl(\ell_m!\bigr)^{k_{m,1}+\cdots+k_{m,d}}}
% \biggr\|_{L^\infty(\Ddom)},
% \end{equation*}
% where
% \begin{align*}
%   I(r,{\bf j}):= \biggl\{&(k_{1,1},\ldots,k_{1,d},k_{2,1},\ldots,k_{2,d},\ldots,k_{r,1},\ldots,k_{r,d};\ell_1,\ldots,\ell_r):\ \\
%   & \text{for some } 1\leq s\leq r,\ k_{m,1}=\cdots=k_{m,d}=\ell_m=0 \text{ for } 1\leq m\leq r-s, \\
%   &k_{m,1}+\cdots+k_{m,d} > 0 \text{ for } r-s+1\leq m\leq r,\
%    \text{and } 0<\ell_{r-s+1}<\cdots <l_r \\
%   & \text{are such that } \\
%   & \sum_{m=1}^r k_{m,1}=j_1,\ \ldots,\ \sum_{m=1}^r k_{m,d}=j_d, \ 
%   \sum_{m=1}^r \ell_m(k_{m,1}+\cdots+k_{m,d})=r \biggr\}  .
% \end{align*}
\begin{equation}\label{eq:constant-geom}
C_{{\bf G},i,r,{\bf j}}:=\biggl\|
\sum_{I(r,{\bf j})} r! \prod_{m=1}^r 
\frac{\bigl(\partial_i^{m}G_1\bigr)^{k_{m,1}}\cdots\bigl(\partial_i^{m}G_d\bigr)^{k_{m,d}}}
{\bigl(k_{m,1}!\cdots k_{m,d}!\bigr) \bigl(m!\bigr)^{k_{m,1}+\cdots+k_{m,d}}}
\biggr\|_{L^\infty(\Ddom)},
\end{equation}
where ${\bf j}:=(j_1,\ldots,j_d)$ and
\begin{align*}
  I(r,{\bf j}):= \biggl\{&(k_{1,1},\ldots,k_{1,d},k_{2,1},\ldots,k_{2,d},\ldots,k_{r,1},\ldots,k_{r,d})\in\ZZ_{\geq0}^{r\times d}: \\
  &\sum_{m=1}^r k_{m,1}=j_1,\ \ldots,\ \sum_{m=1}^r k_{m,d}=j_d, \ 
  \sum_{m=1}^r m(k_{m,1}+\cdots+k_{m,d})=r \biggr\}  .
\end{align*}
\begin{theorem}\label{thm:mapL2}
Let ${\bf G}\in (W^{r,\infty}(\Ddom))^d$ and ${\bf G}^{-1}\in (W^{1,\infty}(\Pdom))^d$. Then, for any $\tilde{u}\in H^r(\Pdom)$ we have
\begin{equation*}
\|\tilde{u}-\widetilde{Z}_{\bf\prec}\tilde{u}\|_{\Pdom}\leq C_{\bf G} \sum_{1\leq |{\bf j}|\leq r}\left(\sum_{i=1}^d\Cgen_{\prec_i,r} C_{{\bf G},i,r,{\bf j}}\right)\|\dpartial_1^{j_1}\cdots\dpartial_d^{j_d}\tilde{u}\|_{\Pdom},
\end{equation*}
for all $\prec_i\geq r-1$.
\end{theorem}
\begin{proof}
By means of the multivariate Fa\`a di Bruno formula in \cite{Constantine:1996} we can express the high-order partial derivatives in Lemma~\ref{lem:smoothG} as
\begin{equation*}
\partial_i^r(\tilde{u}\circ {\bf G})=
\sum_{1\leq |{\bf j}|\leq r}(\dpartial_1^{j_1}\cdots\dpartial_d^{j_d}\tilde{u}) \circ {\bf G}
\,\sum_{I(r,{\bf j})} r! \prod_{m=1}^r 
\frac{\bigl(\partial_i^{m}G_1\bigr)^{k_{m,1}}\cdots\bigl(\partial_i^{m}G_d\bigr)^{k_{m,d}}}
{\bigl(k_{m,1}!\cdots k_{m,d}!\bigr) \bigl(m!\bigr)^{k_{m,1}+\cdots+k_{m,d}}}.
\end{equation*}
This gives
\begin{equation*}
\|\tilde{u}-\widetilde{Z}_{\bf\prec}\tilde{u}\|_{\Pdom}\leq \|\det\nabla {\bf G}\|_{L^\infty(\Ddom)} \sum_{i=1}^d\Cgen_{\prec_i,r}\sum_{1\leq |{\bf j}|\leq r} C_{{\bf G},i,r,{\bf j}}\|(\dpartial_1^{j_1}\cdots\dpartial_d^{j_d}\tilde{u}) \circ {\bf G}\|_{\Ddom},
\end{equation*}
and a standard substitution argument completes the proof.
\end{proof}
In the spirit of Corollary~\ref{cor:tensorL2}, using results from Sections~\ref{sec:low} and \ref{sec:max}, the above theorem can be used to obtain error estimates for mapped $L^2$-projections onto spline spaces of any smoothness. Indeed, we just need to replace $\Cgen_{\prec_i,r}$ with the corresponding constants, e.g., the ones derived in Corollaries~\ref{cor:low-order} and~\ref{cor:smoothL2}.

\begin{example}
Let $d=1$. Given the geometry map $G$, we have
\begin{equation*}
C_{G,1,r,j}=\biggl\|
\sum_{I(r,j)} r! \prod_{m=1}^r 
\frac{\bigl(\partial^{m}G\bigr)^{k_m}}
{\bigl(k_{m}!\bigr) \bigl(m!\bigr)^{k_{m}}}
\biggr\|_{L^\infty(\Ddom)},
\end{equation*}
where
\begin{equation*}
I(r,j):= \biggl\{(k_1,\ldots,k_r) \in \ZZ_{\geq0}^r: \sum_{m=1}^r k_m=j, \ \sum_{m=1}^r m k_m=r \biggr\}.
\end{equation*}
Observe that $C_{G,1,r,j}$ can be compactly expressed in terms of (exponential) partial Bell polynomials $B_{r,j}(x_1,\ldots,x_{r-j+1})$ by
\begin{equation*}
C_{G,1,r,j}=\|B_{r,j}(\partial G,\partial^2 G,\ldots,\partial^{r-j+1}G)\|_{L^\infty(\Ddom)};
\end{equation*}
see, e.g., \cite[Section~3.3]{Comtet:1974}.
These Bell polynomials can be easily computed by the following recurrence relation:
% \begin{equation*}
% B_{r,j}(x_1,\ldots,x_{r-j+1}) = \sum_{i=1}^{r-j+1}\binom{r-1}{i-1}x_i B_{r-i,j-1}(x_1,\ldots,x_{r-i-j+2}),
% \end{equation*}
\begin{equation*}
B_{r,j}(x_1,\ldots,x_{r-j+1}) = \frac{1}{j}\sum_{i=j-1}^{r-1}\binom{r}{i}x_{r-i} B_{i,j-1}(x_1,\ldots,x_{i-j+2}),
\end{equation*}
where $B_{0,0}=1$ and $B_{r,0}=0$ for $r\geq1$. %, and $B_{0,j}=0$ for $j\geq1$.
In particular, we have
\begin{align*}
B_{1,1}(x_1) &=x_1, \\
B_{2,1}(x_1,x_2) &=x_2, \quad  B_{2,2}(x_1)=(x_1)^2, \\
B_{3,1}(x_1,x_2,x_3) &=x_3, \quad  B_{3,2}(x_1,x_2)=3x_1x_2, \quad B_{3,3}(x_1)=(x_1)^3.
\end{align*}
\end{example}
\begin{example}\label{ex:mapL2,r=2}
Let $d=2$. For $r=1$ and $i=1,2$ we have
\begin{equation*}
C_{{\bf G},i,1,(1,0)}=\| \partial_i G_1 \|_{L^\infty(\Ddom)}, \quad
C_{{\bf G},i,1,(0,1)}=\| \partial_i G_2 \|_{L^\infty(\Ddom)}.
\end{equation*}
For $r=2$ and $i=1,2$ we have
\begin{align*}
C_{{\bf G},i,2,(1,0)}&=\| \partial_i^2 G_1 \|_{L^\infty(\Ddom)}, \quad
C_{{\bf G},i,2,(0,1)}=\| \partial_i^2 G_2 \|_{L^\infty(\Ddom)}, \\
C_{{\bf G},i,2,(2,0)}&=\| (\partial_i G_1)^2 \|_{L^\infty(\Ddom)}, \quad
C_{{\bf G},i,2,(0,2)}=\| (\partial_i G_2)^2 \|_{L^\infty(\Ddom)}, \\
C_{{\bf G},i,2,(1,1)}&=\| 2(\partial_i G_1)(\partial_i G_2) \|_{L^\infty(\Ddom)}.
\end{align*}
\end{example}

Similar results can be obtained for tensor-product Ritz projections in the presence of a mapped geometry. As before, it is a matter of applying the Ritz estimates from Section~\ref{sec:tensor} in combination with the multivariate Fa\`a di Bruno formula \cite{Constantine:1996}. We omit these results to avoid repetition.
We just illustrate this with an example.
\begin{example}\label{ex:mapRitz,r=2}
Let $d=2$ and $r=2$.
Recall from Example~\ref{ex:tensorRitz-smooth} that for any $u\in H^2(\Ddom)$ and for $0\leq\ell_1,\ell_2\leq1$ we have
\begin{equation*}
\|\partial_1^{\ell_1}\partial_2^{\ell_2}(u-R_{\bf p}u)\|_{\Ddom}\leq \left(\frac{h}{\pi}\right)^{2-\ell_1-\ell_2}\left(\|\partial_1^2u\|_{\Ddom}+\|\partial_2^2u\|_{\Ddom}+\|\partial_1\partial_2u\|_{\Ddom}\right),
\end{equation*}
for all $p_1,p_2\geq 1$ and $h:=\max\{h_1,h_2\}$.
We define the mapped Ritz projector
$\widetilde{R}_{\bf p}: H^2(\Pdom)\to \widetilde{\mathcal{Z}}_{\bf\prec}$ 
by taking $\Pi=R_{\bf p}$ in \eqref{eq:mappedProj}.
Assume ${\bf G}\in (W^{2,\infty}(\Ddom))^2$ and ${\bf G}^{-1}\in (W^{1,\infty}(\Pdom))^2$.
From Theorem~\ref{thm:mapL2} (and Example~\ref{ex:mapL2,r=2}) we know estimates for $\|\partial_1^2(\tilde{u}\circ {\bf G})\|_{\Ddom}$ and $\|\partial_2^2(\tilde{u}\circ {\bf G})\|_{\Ddom}$, and we can compute similar ones for 
$\|\partial_1\partial_2(\tilde{u}\circ {\bf G})\|_{\Ddom}$.
Then, for any $\tilde{u}\in H^2(\Pdom)$ and for $0\leq\ell_1,\ell_2\leq1$ we obtain
\begin{equation*}
\|\dpartial_1^{\ell_1}\dpartial_2^{\ell_2}(\tilde{u}-\widetilde{R}_{\bf p}\tilde{u})\|_{\Pdom}\leq C_{\bf G} \left(\frac{h}{\pi}\right)^{2-\ell_1-\ell_2} \sum_{1\leq |{\bf j}|\leq 2} \bigl(C_{{\bf G},1,2,{\bf j}}+C_{{\bf G},2,2,{\bf j}}+C_{{\bf G},12,2,{\bf j}}\bigr) \|\dpartial_1^{j_1}\dpartial_2^{j_2}\tilde{u}\|_{\Pdom},
\end{equation*}
for all $p_1,p_2\geq 1$, where
\begin{align*}
C_{{\bf G},12,2,(1,0)}&=\| \partial_1\partial_2 G_1 \|_{L^\infty(\Ddom)}, \quad
C_{{\bf G},12,2,(0,1)}=\| \partial_1\partial_2 G_2 \|_{L^\infty(\Ddom)}, \\
C_{{\bf G},12,2,(2,0)}&=\| (\partial_1 G_1)(\partial_2 G_1) \|_{L^\infty(\Ddom)}, \quad
C_{{\bf G},12,2,(0,2)}=\| (\partial_1 G_2)(\partial_2 G_2) \|_{L^\infty(\Ddom)}, \\
C_{{\bf G},12,2,(1,1)}&=\| (\partial_1 G_1)(\partial_2 G_2)+(\partial_2 G_1)(\partial_1 G_2) \|_{L^\infty(\Ddom)}.
\end{align*}
\end{example}

\subsection{Bent geometry}\label{subsec:bent}
In IGA the geometry map ${\bf G}$ is commonly taken to be componentwise a spline function from the same space as our approximation space. However, the results in the previous subsection can require the geometry map to be in a smoother subspace. We will overcome the issue in this subsection. As before, we use the techniques of \cite{Bazilevs:2006,Buffa:14}.
%In this section we only assume ${\bf G}\in (W^{1,\infty}(\Ddom))^d$ and ${\bf G}^{-1}\in (W^{1,\infty}(\Pdom))^d$.

For $r\geq 1$ and $k\geq 0$ we define the univariate bent Sobolev space 
\begin{equation*}
\mathcal{H}^{r,k}_{\knots}(0,1):= \{u\in H^{\min\{r,k+1\}}(0,1): 
u \in H^r(\xi_j,\xi_{j+1}),\, j=0,1,\dots,\nknots\}.
\end{equation*}
Note that for $k\geq r-1$ we have $\mathcal{H}^{r,k}_{\knots}(0,1)=H^{r}(0,1)$.
Then, similar to the space in Remark~\ref{rem:tensorL2}, we define the ($L^2$-extended) multivariate bent Sobolev space
\begin{equation*}
\mathcal{H}^{r,{\bf k}}_{\dknots}(\Ddom):= 
%\mathcal{H}^{r_1,k_1}(0,1)\otimes\dots  \mathcal{H}^{r_d,k_d}(0,1).
\bigcap_{i=1}^d L^2(0,1)^{i-1}\otimes\mathcal{H}^{r,k_i}_{\knots_i}(0,1)\otimes L^2(0,1)^{d-i}.
\end{equation*}
%where ${\bf k}:=(k_1,k_2,\ldots,k_d)$.
Following \cite{Buffa:14} we also introduce the mesh-dependent norm 
\begin{equation*}
\|\cdot\|^2_{\Ddom,\dknots}:=\sum_{\sigma\in M_{\dknots}}\|\cdot\|^2_{\sigma},
\end{equation*}
where $M_{\dknots}$ is the collection of the (open) elements defined by $\dknots$ and $\|\cdot\|_{\sigma}$ denotes the $L^2$-norm on the element $\sigma$.

Furthermore, for $k_i\geq 0$, $i=1,\ldots,d$, we define the bent geometry function class
\begin{equation*}
\mathcal{G}^{r,{\bf k}}_{\dknots}(\Ddom) := \{G \in W^{{\bf k+1},\infty}(\Ddom): G\in W^{r,\infty}(\sigma),\,\sigma\in M_{\dknots} \},
\end{equation*}
where $W^{{\bf k+1},\infty}(\Ddom):=W^{k_1+1,\infty}(0,1)\otimes \cdots\otimes W^{k_d+1,\infty}(0,1)$. 
The space $\mathcal{G}^{r,{\bf k}}_{\dknots}(\Ddom)$ contains the spline space $\mathcal{S}^{\bf k}_{{\bf p},\dknots}$, and it also allows for several other interesting piecewise spaces such as NURBS spaces based on $\mathcal{S}^{\bf k}_{{\bf p},\dknots}$.
If we assume ${\bf G}\in (\mathcal{G}^{r,{\bf k}}_{\dknots}(\Ddom))^d$, then $u:=\tilde{u}\circ {\bf G} \in \mathcal{H}^{r,{\bf k}}_{\dknots}(\Ddom)$ for $\tilde{u}\in H^r(\Pdom)$. 
Having $u$ not in $H^r(\Ddom)$ is a potential problem for applying the error estimates we derived in Section~\ref{sec:tensor}, but this can be fixed by making use of \cite[Lemma~3.1]{BeiraoDaVeiga:2012}.
For completeness we provide a short proof here as well.
\begin{lemma}\label{lem:Gamma}
For $k\leq r-2$, there exists an operator $\Gamma: \mathcal{H}^{r,k}_{\knots}(0,1)\to \mathcal{S}^{k}_{r-1,\knots}$ such that $u-\Gamma u\in H^r(0,1)$ for all $u\in \mathcal{H}^{r,k}_{\knots}(0,1)$.
\end{lemma}
\begin{proof}
Let $u\in \mathcal{H}^{r,k}_{\knots}(0,1)$ for some $k\leq r-2$, 
and let $\partial^\ell_-u$ ($\partial^\ell_+u$) denote the limit from the left (right) of the $\ell$-th order derivative of $u$.
From the definition of the bent Sobolev space we know that $u$ is $C^k$ continuous at any interior knot $\xi_j$, $j=1,\ldots,\nknots$. 
%In the case $k=-1$ we allow for jumps at these interior knots. 
Now, we define
\begin{equation*}
  \varphi_{j,k}(x):= \frac{(\partial^{k+1}_+u -\partial^{k+1}_-u)(\xi_j)}{(k+1)!}\max\{0,(x-\xi_j)^{k+1}\}.
\end{equation*}
It is easy to check that $\varphi_{j,k}\in\mathcal{S}^{k}_{r-1,\knots}$ and that
$u-\varphi_{j,k}$ is $C^{k+1}$ continuous at the knot $\xi_j$. Repeating this argument and taking
\begin{equation*}
\Gamma u = \sum_{j=1}^{\nknots} \sum_{l=k}^{r-2} \varphi_{j,l},
\end{equation*}
it follows that $u-\Gamma u$ is $C^{r-1}$ continuous at each interior knot. Since $\Gamma u\in\mathcal{S}^{k}_{r-1,\knots}$ we also know that $u-\Gamma u\in \mathcal{H}^{r,k}_{\knots}(0,1)$, and so $u-\Gamma u\in H^r(0,1)$.
\end{proof}

Similar to \cite[Proposition~3.1]{BeiraoDaVeiga:2012} we then obtain the following error estimate.
\begin{lemma}\label{lem:broken}
Let $u\in \mathcal{H}^{r,k}_{\knots}(0,1)$ be given. 
%For any knot sequence $\knots$, let $S^k_p$ be the $L^2$-projection onto $\mathcal{S}^{k}_{p,\knots}$ for $-1\leq k\leq p-1$. 
Then,
\begin{equation*}
\|u-S^k_pu\|\leq C_{h,p,k,r}\|\partial^ru\|_{(0,1),\knots},
\end{equation*}
for all $p\geq r-1$.
\end{lemma}
\begin{proof}
For $k\geq r-1$, the result immediately follows from Corollary~\ref{cor:low-order} by recalling that $\mathcal{H}^{r,k}_{\knots}(0,1)=H^r(0,1)$ in this case. Assume now $k\leq r-2$.
Since $\mathcal{S}^{k}_{r-1,\knots}\subseteq \mathcal{S}^{k}_{p,\knots}$ we deduce from Lemma~\ref{lem:Gamma} and Corollary~\ref{cor:low-order} that
\begin{align*}
\|u-S^k_pu\|^2&=\|u-\Gamma u-S^k_p(u-\Gamma u)\|^2 \leq \left(C_{h,p,k,r}\|\partial^r(u-\Gamma u)\|\right)^2
\\
&= (C_{h,p,k,r})^2\sum_{j=0}^{\nknots}\|\partial^ru\|^2_{(\xi_j,\xi_{j+1})} = \left(C_{h,p,k,r}\|\partial^ru\|_{(0,1),\knots}\right)^2,
\end{align*}
and the result follows by taking the square root of both sides.
\end{proof}
% If $k=p-1$ then the constant $C_{h,p,k,r}$ in the above lemma can be replaced by the sharper constant $C_{h,p,r}$ in Section \ref{sec:max}.

The univariate error estimate in Lemma~\ref{lem:broken} can be easily extended to the multivariate tensor-product spline setting.
\begin{lemma}\label{lem:broken-tensor}
Let $u\in \mathcal{H}^{r,{\bf k}}_{\dknots}(\Ddom)$ be given. Then,
\begin{align*}
\|u- S^{\bf k}_{\bf p}u\|_{\Ddom}&\leq \sum_{i=1}^dC_{h_i,p_i,k_i,r}\|\partial_i^ru\|_{\Ddom,\dknots},
% \\
% \|u- S_{\bf p}u\|_{\Ddom}&\leq \sum_{i=1}^dC_{h_i,p_i,r}\|\partial_i^ru\|_{\Ddom,\dknots},
\end{align*}
for all $p_i\geq r-1$.
\end{lemma}
\begin{proof}
Using Theorem~\ref{thm:tensorL2} we have 
\begin{equation*}
\|u- S^{\bf k}_{\bf p}u\|_{\Ddom} \leq \sum_{i=1}^d\|u- S^{k_i}_{p_i}u\|_{\Ddom},
\end{equation*}
and the result follows by applying Lemma~\ref{lem:broken} in each direction separately.
\end{proof}

In the case of maximal spline smoothness, i.e., $k_i=p_i-1$ for all $i$, the constants $C_{h_i,p_i,k_i,r}$ in the above lemma can be replaced by the constants used in Corollary~\ref{cor:smoothL2}.

Using the argument of Theorem~\ref{thm:mapL2} we then arrive at the desired error estimates for a bent geometry.
To this end, we need to redefine the constants $C_{{\bf G},i,r,{\bf j}}$ in \eqref{eq:constant-geom} using the mesh-dependent norm
\begin{equation}\label{eq:inf-norm-mesh}
\|\cdot\|_{L^\infty(\Ddom),\dknots}:=\max_{\sigma\in M_{\dknots}}\|\cdot\|_{L^\infty(\sigma)}.
\end{equation}
\begin{theorem}\label{thm:mapL2-gen}
Let ${\bf G}\in (\mathcal{G}^{r,{\bf k}}_{\dknots}(\Ddom))^d$ and ${\bf G}^{-1}\in (W^{1,\infty}(\Pdom))^d$. 
Then, for any $\tilde{u}\in H^r(\Pdom)$ we have
\begin{equation*}
\|\tilde{u}-\widetilde{S}^{\bf k}_{\bf p}\tilde{u}\|_{\Pdom} \leq C_{\bf G} \sum_{1\leq |{\bf j}|\leq r}\left(\sum_{i=1}^dC_{h_i,p_i,k_i,r} C_{{\bf G},i,r,{\bf j}}\right)\|\dpartial_1^{j_1}\cdots\dpartial_d^{j_d}\tilde{u}\|_{\Pdom},
\end{equation*}
for all $p_i\geq r-1$. 
% Moreover, in case of maximal spline smoothness we have
% \begin{equation*}
% \|\tilde{u}-\widetilde{S}_{\bf p}\tilde{u}\|_{\Pdom} \leq C_{\bf G} \sum_{1\leq |{\bf j}|\leq r}\left(\sum_{i=1}^dC_{h_i,p_i,r} C_{{\bf G},i,r,{\bf j}}\right)\|\dpartial_1^{j_1}\cdots\dpartial_d^{j_d}\tilde{u}\|_{\Pdom},
% \end{equation*}
% for all $p_i\geq r-1$.
\end{theorem}

Similar results can be obtained for tensor-product Ritz projections in the presence of a bent geometry. As before, it is a matter of applying the Ritz estimates from Section~\ref{sec:tensor} in combination with the operator in Lemma~\ref{lem:Gamma} and proper (Ritz extended) multivariate bent Sobolev spaces.
We omit these results to avoid repetition.
We just illustrate this with an example similar to Example~\ref{ex:mapRitz,r=2}.
\begin{example}\label{ex:splinemapRitz,r=2}
Let $d=2$ and $r=2$. Assuming
${\bf G}\in (\mathcal{S}_{{\bf p},\dknots})^2$ and ${\bf G}^{-1}\in (W^{1,\infty}(\Pdom))^2$,
for any $\tilde{u}\in H^2(\Pdom)$ and for $0\leq\ell_1,\ell_2\leq1$ we have
\begin{equation*}
\|\dpartial_1^{\ell_1}\dpartial_2^{\ell_2}(\tilde{u}-\widetilde{R}_{\bf p}\tilde{u})\|_{\Pdom}\leq C_{\bf G} \left(\frac{h}{\pi}\right)^{2-\ell_1-\ell_2} \sum_{1\leq |{\bf j}|\leq 2} \bigl(C_{{\bf G},1,2,{\bf j}}+C_{{\bf G},2,2,{\bf j}}+C_{{\bf G},12,2,{\bf j}}\bigr) \|\dpartial_1^{j_1}\dpartial_2^{j_2}\tilde{u}\|_{\Pdom},
\end{equation*}
for all $p_1,p_2\geq 1$ and $h:=\max\{h_1,h_2\}$. The constants in the above sum are the same as the ones in Examples~\ref{ex:mapL2,r=2} and~\ref{ex:mapRitz,r=2} but in the mesh-dependent norm \eqref{eq:inf-norm-mesh}.
\end{example}

%%%%%%%%%%%%%
\section{Multi-patch geometry} \label{sec:multipatch}
%%%%%%%%%%%%%
In this section we generalize our error estimates to the case of multi-patch domains with $C^0$ continuity across the patches. The arguments here are based on those found in \cite{Buffa:14,Takacs:2018}.

We start by explaining the general framework in the univariate case.
Let $\mathcal{Z}_\prec$ be a finite dimensional subspace of $L^2(a,b)$ as in \eqref{eq:Xsimpl} with $K$ as in \eqref{eq:Kint}.
For $\prec\geq 1$ we define the projector $Q_{\prec}:H^1(a,b)\to \mathcal{Z}_{\prec}$ by
\begin{equation}\label{def:Qproj}
Q_{\prec}u:= u(a) + KZ_{\prec-1}\partial u,
\end{equation}
where $K$ is the integral operator in \eqref{eq:Kint} and $Z_{\prec}$ the $L^2$-projector onto $\mathcal{Z}_{\prec}$. As we shall see momentarily, the projection \eqref{def:Qproj} is closely related to the Ritz projection for $q=1$ in \eqref{def:Ritz} and satisfies essentially the same properties. Additionally, we observe that $Q_{\prec}u(a) = u(a)$ and
\begin{equation}\label{eq:Q-bnd}
Q_{\prec}u(b)= u(a) + \int_a^bZ_{\prec-1}\partial u(x) dx = u(a) + \int_a^b\partial u(x) dx = u(b).
\end{equation}
Thus, $Q_{\prec}$ can be equivalently expressed as 
\begin{equation}\label{def:Qproj2}
Q_{\prec}u= u(b) - K^*Z_{\prec-1}\partial u.
\end{equation}
The interpolation at the boundary will be used to enforce $C^0$ continuity across the patches.
 Similar to the case $q=1$ of Theorem~\ref{thm:Ritz} we have the following error estimate.
\begin{lemma}\label{lem:errorQ}
Let $u\in H^r(a,b)$ for $r\geq 1$ be given. Then, 
\begin{align*}
\|u-Q_\prec u\|&\leq  \Cgen_{\prec-1,1} \Cgen_{\prec-1,r-1}\|\partial^ru\|,
\\
\|\partial(u-Q_\prec u)\|&\leq  \Cgen_{\prec-1,r-1}\|\partial^ru\|,
\end{align*}
for all $\prec\geq 1$ such that $\mathcal{P}_{r-2}\subseteq\mathcal{Z}_{\prec-1}$.
\end{lemma}
\begin{proof}
By the fundamental theorem of calculus we have $u=u(b) - K^*v$ for $v\in H^{r-1}(a,b)$. Thus, using \eqref{def:Qproj2},
\begin{equation*}
\|u-Q_\prec u\| = \|K^*v - K^*Z_{\prec-1}v\| = \|K^*(I-Z_{\prec-1})v\|.
\end{equation*}
Moreover, $v\in H^{r-1}(a,b)$ can be written as $v=g+K^{r-1}f$ for $g\in \mathcal{P}_{r-2}$ and $f\in L^2(a,b)$. Using $\mathcal{P}_{r-2}\subseteq\mathcal{Z}_{\prec-1}$ and $(I-Z_{\prec-1})^2=(I-Z_{\prec-1})$ we obtain 
\begin{align*}
\|K^*(I-Z_{\prec-1})v\|&= \|K^*(I-Z_{\prec-1})K^{r-1}f\|\leq \|K^*(I-Z_{\prec-1})\|\,\|(I-Z_{\prec-1})K^{r-1}\|\,\|f\|
\\
&=\|(I-Z_{\prec-1})K\|\,\|(I-Z_{\prec-1})K^{r-1}\|\,\|f\| = \Cgen_{\prec-1,1} \Cgen_{\prec-1,r-1}\|\partial^ru\|,
\end{align*}
which proves the first inequality. The second inequality follows from Theorem \ref{thm:L2} since $\partial Q_\prec = Z_{\prec-1}\partial$.
\end{proof}
Error estimates for spline spaces can be immediately obtained by replacing the constants in Lemma~\ref{lem:errorQ} with the constants derived for $q=1$ in Corollaries~\ref{cor:Ritz-lower} and~\ref{cor:smoothRitz}.

\begin{remark}
Since the Ritz projection in \eqref{def:Ritz} is uniquely defined, it follows from Lemma~\ref{lem:Stefan} and \eqref{eq:Q-bnd} that $Q_{\prec}=R^1_\prec$ whenever $\mathcal{P}_2\subseteq\mathcal{Z}_\prec$. Lemma~\ref{lem:errorQ} would in this case directly follow from Lemma~\ref{lem:Stefan} and Theorem~\ref{thm:Ritz}.
\end{remark}

We now move on to the bivariate case ($d=2$).
As before, we let ${\bf\prec}=(\prec_1,\prec_2)$ and define the tensor-product projector $Q_{\bf\prec}: H^1(a_1,b_1)\otimes H^1(a_2,b_2)\to\mathcal{Z}_{\prec_1}\otimes \mathcal{Z}_{\prec_2}$ by
\begin{equation*}
Q_{\bf\prec} := Q_{\prec_1}\otimes Q_{\prec_2}.
\end{equation*}
\begin{remark}\label{rmk:Q-tens-bnd}
As in \cite[Theorem~3.4]{Takacs:2018}, we conclude from \eqref{def:Qproj} and \eqref{eq:Q-bnd} that for all $u\in H^1(a_1,b_1)\otimes H^1(a_2,b_2)$,
\begin{itemize}
\item $u$ and $Q_{\bf\prec}u$ coincide at the four corners of $[a_1,b_1]\times[a_2,b_2]$, and
\item $Q_{\bf\prec}u$ restricted to any boundary edge of $\Ddom=(a_1,b_1)\times(a_2,b_2)$ coincide with the univariate projection onto that edge, e.g.,
\begin{equation*}
Q_{\bf\prec}u(a_1,\cdot) = Q_{\prec_2}u(a_1,\cdot).
\end{equation*}
\end{itemize}
\end{remark}
Using the same argument as for Theorem~\ref{thm:tensorRitz} we obtain the following error estimates for~$Q_{\bf\prec}$.
\begin{theorem}
Let $u\in H^r(\Ddom)$ for $r\geq2$ be given. If $\mathcal{P}_{r-2}\subseteq \mathcal{Z}_{\prec_1-1}\cap\mathcal{Z}_{\prec_2-1}$ for $\prec_1,\prec_2\geq1$, then
\begin{align*}
\|u-Q_{\bf \prec}u\|_{\Ddom}&\leq \Cgen_{\prec_1-1,1}\Cgen_{\prec_1-1,r-1}\|\partial_1^ru\|_{\Ddom}+\Cgen_{\prec_2-1,1}\Cgen_{\prec_2-1,r-1}\|\partial_2^ru\|_{\Ddom} 
\\
&\quad+\Cgen_{\prec_1-1,1}\Cgen_{\prec_2-1,1}\min\left\{\Cgen_{\prec_2-1,r-2}\|\partial_1\partial_2^{r-1}u\|_{\Ddom},\Cgen_{\prec_1-1,r-2}\|\partial_1^{r-1}\partial_2u\|_{\Ddom}\right\},
\end{align*}
and
\begin{align*}
\|\partial_1(u-Q_{\bf \prec}u)\|_{\Ddom}&\leq \Cgen_{\prec_1-1,r-1}\|\partial_1^ru\|_{\Ddom}+\Cgen_{\prec_2-1,1}\Cgen_{\prec_2-1,r-2}\|\partial_1\partial_2^{r-1}u\|_{\Ddom}, 
\\
\|\partial_2(u-Q_{\bf \prec}u)\|_{\Ddom}&\leq \Cgen_{\prec_1-1,1}\Cgen_{\prec_1-1,r-2}\|\partial_1^{r-1}\partial_2u\|_{\Ddom}+\Cgen_{\prec_2-1,r-1}\|\partial_2^ru\|_{\Ddom},
\\
\|\partial_1\partial_2(u-Q_{\bf \prec}u)\|_{\Ddom}&\leq  \Cgen_{\prec_1-1,r-2}\|\partial_1^{r-1}\partial_2u\|_{\Ddom}+\Cgen_{\prec_2-1,r-2}\|\partial_1\partial_2^{r-1}u\|_{\Ddom}. 
\end{align*}
\end{theorem}
In the spirit of Corollary~\ref{cor:tensorL2}, using results from Sections~\ref{sec:low} and \ref{sec:max}, the above theorem can be used to obtain similar error estimates for spline spaces of any smoothness.

Finally, we are ready to consider the multi-patch setting in IGA. We 
%restrict ourselves to the case $\prec_1=\prec_2$ in $\bf\prec$ and 
assume that the physical domain $\Pdom\subset\RR^2$ is divided into $\npatches$ non-overlapping patches $\Pdom_i$, $i=1,\ldots, \npatches$. The patches are conforming, i.e., the intersection of the closures of $\Pdom_i$ and $\Pdom_j$ for $i\neq j$ is either (a) empty, (b) one common corner, or (c) the union of one common edge and two common vertices.
Following \cite{Takacs:2018}, we define the bent Sobolev space in the physical domain $\mathcal{H}^{2,1}(\Pdom)$ by
\begin{equation*}
\mathcal{H}^{2,1}(\Pdom):=\{\tilde{u}\in H^1(\Pdom) : \tilde{u}|_{\Pdom_i}\in H^2(\Pdom_i), \, i=1,\ldots,\npatches\}.
\end{equation*}

We assume that for each $i=1,\ldots, \npatches$ there is a geometry map ${\bf G}_i: \Ddom=(0,1)^2 \to \Pdom_i$, which can be continuously extended to the closure of $\Ddom$, such that
\begin{itemize}
\item  ${\bf G}_i\in (\mathcal{G}^{r,{\bf k}}_{\dknots}(\Ddom))^2$ and ${\bf G}_i^{-1}\in (W^{1,\infty}(\Pdom_i))^2$ (see Section~\ref{subsec:bent}), and

\item for any interface $\widetilde\Gamma_{ij}$ shared by $\Pdom_i$ and $\Pdom_j$, the parameterizations ${\bf G}_i$ and ${\bf G}_j$ are identical along that interface, i.e.,
${\bf G}_i^{-1}|_{\widetilde\Gamma_{ij}}={\bf R}_{ij}\circ{\bf G}_j^{-1}|_{\widetilde\Gamma_{ij}}$
where  ${\bf R}_{ij}$ is a rigid motion of the unit square to itself.
\end{itemize}
Similar to \eqref{eq:mappedSpace} we define 
\begin{equation*}
\widetilde{\mathcal{Z}}_{{\bf \prec},i}:=\{s\circ {\bf G}_i^{-1}: s\in \mathcal{Z}_{{\bf \prec},i}\},
\end{equation*}
and, following \cite{Buffa:14,Takacs:2018}, we require that these function spaces are fully matching on the interfaces, i.e.,
for each $\tilde{s}_i\in\widetilde{\mathcal{Z}}_{{\bf \prec},i}$ there exists $\tilde{s}_j\in\widetilde{\mathcal{Z}}_{{\bf \prec},j}$ 
such that along any interface $\widetilde\Gamma_{ij}$ shared by $\Pdom_i$ and $\Pdom_j$ we have
\begin{equation*}
{\tilde{s}}_i|_{\widetilde\Gamma_{ij}}={\tilde{s}}_j|_{\widetilde\Gamma_{ij}}.
\end{equation*}
\begin{remark} 
Under the assumptions on the geometry maps, the fully matching requirement at the interface $\widetilde\Gamma_{ij}$ is simply satisfied whenever for $l=i,j$
the univariate spaces ${\mathcal{Z}}_{\prec_{m_l},l}$, $ m_l\in\{1,2\}$, associated with ${\bf G}_l^{-1}({\widetilde\Gamma_{ij}})$ coincide. For instance, if ${\bf G}_i^{-1}({\widetilde\Gamma_{ij}})$ is a horizontal edge while ${\bf G}_j^{-1}({\widetilde\Gamma_{ij}})$ is a vertical one, then
${\mathcal{Z}}_{\prec_1,i}= {\mathcal{Z}}_{\prec_2,j}$.
\end{remark}
With the patch spaces $\widetilde{\mathcal{Z}}_{{\bf\prec},i}$ in place, we define the continuous isogeometric multi-patch space 
$\widetilde{\mathcal{Z}}_{\bf\prec}:\Pdom\to \RR $ as the continuously glued collection of those patch spaces, i.e.,
\begin{equation*}
\widetilde{\mathcal{Z}}_{\bf \prec}:=\{\tilde s\in C^0(\Pdom) : {\tilde s}|_{\Pdom_i}\in \widetilde{\mathcal{Z}}_{{\bf \prec},i}, \, i=1,\ldots, M  \}.
\end{equation*}
We let $\widetilde{Q}_{{\bf\prec}, i}:H^2(\Pdom_i) \to \widetilde{\mathcal{Z}}_{{\bf \prec},i}$
denote the projector defined by
\begin{equation*}
\widetilde{Q}_{{\bf\prec}, i} \tilde{u} := (Q_{{\bf\prec}, i}(\tilde{u}\circ {\bf G}_i))\circ {\bf G}_i^{-1},
\quad\forall \tilde{u}\in H^2(\Pdom_i),
\end{equation*} 
and for any ${\tilde u}\in \mathcal{H}^{2,1}(\Pdom)$
we define $\widetilde{Q}_{\bf\prec} (\tilde u)$ by
\begin{align*}
(\widetilde{Q}_{\bf\prec} \tilde u)|_{\Pdom_i} := \widetilde{Q}_{{\bf\prec}, i} \tilde u.
\end{align*}
With the same line of arguments as in \cite[Proposition~3.8]{Buffa:14} (see also  \cite[Lemma~3.4]{Takacs:2018}), by using Remark~\ref{rmk:Q-tens-bnd} together with the requirement that the patch spaces are fully matching, it follows that $\widetilde{Q}_{\bf\prec}{\tilde u}$ can be extended to a continuous function across the patch-interfaces and hence this is a projector onto $\widetilde{\mathcal{Z}}_{\bf\prec}$.

Similar to the mapped Ritz projection in the previous section we can now obtain error estimates for the projector $\widetilde{Q}_{\bf\prec}$. As a continuation of Example~\ref{ex:splinemapRitz,r=2} we can for instance obtain the following result.

\begin{example}\label{ex:splinemapQ,r=2}
Let $d=2$ and $r=2$. Assume
${\bf G}_i\in (\mathcal{S}_{{\bf p},\dknots})^2$ and ${\bf G}_i^{-1}\in (W^{1,\infty}(\Pdom_i))^2$ for $i=1,\ldots,\npatches$. Then,
for any $\tilde{u}\in H^2(\Pdom_i)$ and for $0\leq\ell_1,\ell_2\leq1$ we have
\begin{equation}\label{ineq:multipatch-ex}
\|\dpartial_1^{\ell_1}\dpartial_2^{\ell_2}(\tilde{u}-\widetilde{Q}_{\bf p}\tilde{u})\|_{\Pdom_i}\leq C_{{\bf G}_i} \left(\frac{h}{\pi}\right)^{2-\ell_1-\ell_2} \sum_{1\leq |{\bf j}|\leq 2} \bigl(C_{{\bf G}_i,1,2,{\bf j}}+C_{{\bf G}_i,2,2,{\bf j}}+C_{{\bf G}_i,12,2,{\bf j}}\bigr) \|\dpartial_1^{j_1}\dpartial_2^{j_2}\tilde{u}\|_{\Pdom_i},
\end{equation}
for all $p_1,p_2\geq 1$ and $i=1,\ldots, \npatches$. Here $h:=\max\{h_1,h_2\}$. The constants in the above estimate are the same as the ones in Example~\ref{ex:splinemapRitz,r=2}. By squaring both sides of \eqref{ineq:multipatch-ex} and summing over all the patches one can obtain a global estimate for $\tilde{u}\in\mathcal{H}^{2,1}(\Pdom)$.
\end{example}

\begin{remark}
If $\tilde{u}\in\mathcal{H}^{2,1}(\Pdom)$ is zero at the boundary then it follows from Remark~\ref{rmk:Q-tens-bnd} and the definition of $\widetilde{Q}_{\bf\prec}$ that $\widetilde{Q}_{\bf\prec}{\tilde u}$ is also zero at the boundary. Thus, we can obtain the same error estimates in the case of Dirichlet boundary conditions.
\end{remark}

%%%%%%%%%%%%%
\section{Conclusions} \label{sec:conclusion}
%%%%%%%%%%%%%
In this paper we have provided a priori error estimates with explicit constants for approximation with both classical spline spaces and with their isogeometric extensions. More precisely, we have considered error estimates in Sobolev (semi-)norms for $L^2$ and Ritz projections of any function in $H^r$ onto univariate and multivariate spline spaces, addressing single-patch and $C^0$ multi-patch configurations.~~

In order to obtain these estimates we have introduced an abstract framework to convert explicit constants in polynomial approximation to explicit constants in spline approximation of arbitrary smoothness and on arbitrary knot sequences. 
The constants in our spline error estimates are not sharp as they stem from constants in global polynomial approximation that are not sharp. However, our abstract framework is independent of the polynomial error estimate we start with. Whenever better constants are available for polynomial approximation, they can be simply plugged into our framework, resulting immediately in a sharper result for spline approximation.

Our results improve upon existing error estimates in the literature
as they fill the gap of the smoothness \cite{Buffa:11} and allow for more flexible $h$-$p$ refinement for spline spaces of maximal smoothness \cite{Takacs:2016,Sande:2019}. Moreover, they are consistent with the numerical evidence that smoother spline spaces exhibit a better approximation behavior per degree of freedom, which has been observed when solving practical problems by the IGA paradigm. Our error estimates also pave the way for extending to arbitrary smoothness and to arbitrary knot sequences the theoretical comparison, recently performed in \cite{Bressan:2019}, of the approximation power of different piecewise polynomial spaces commonly employed in Galerkin methods for solving partial differential equations.
In case of a mapped domain, the error estimates explicitly highlight the influence of the (derivatives of the) geometry map on the approximation properties of the considered isogeometric spaces.

Besides their direct theoretical interest, the presented results may have an %important 
impact on several practical aspects of the IGA paradigm, including the convergence analysis under different kinds of refinements, the definition of good mesh quality metrics, and the design of fast iterative (multigrid) solvers for the resulting linear systems.
We finally note that the range of possible applications of the presented results is not confined to the IGA context, since standard $C^0$ tensor-product finite elements are also covered as special cases.

\section*{Acknowledgements}
The authors are very grateful to Stefan Takacs (RICAM, Austria) for pointing out Lemma~\ref{lem:Stefan}.
This work was supported 
by the Beyond Borders Programme of the University of Rome Tor Vergata through the project ASTRID (CUP E84I19002250005)
and
by the MIUR Excellence Department Project awarded to the Department of Mathematics, University of Rome Tor Vergata (CUP E83C18000100006).
C.~Manni and H.~Speleers are members of Gruppo Nazionale per il Calcolo Scientifico, Istituto Nazionale di Alta Matematica.

%%%%%%%%%%%%%%%%%%%%
\bibliography{nwidths}
%%%%%%%%%%%%%%%%%%%%

\end{document}